\documentclass[11pt]{article}
\usepackage{caption}
\usepackage[utf8]{inputenc}
\usepackage{amsmath}
\usepackage{amssymb}
\usepackage{times}
\usepackage{amsthm}
\usepackage{makecell}

\newcommand{\RN}[1]{%
\textup{\uppercase\expandafter{\romannumeral#1}}%
}

\newtheorem{theorem}{Theorem}[section]
\newtheorem{corollary}[theorem]{Corollary}
\newtheorem{lemma}[theorem]{Lemma}
\newtheorem{definition}[theorem]{Definition}
\newtheorem{remark}[theorem]{Remark}
\usepackage[table]{xcolor}
\usepackage{float}
\restylefloat{table}
\usepackage{dblfloatfix}
\usepackage[hidelinks]{hyperref}
\usepackage{xcolor}
\hypersetup{
    colorlinks,
    linkcolor={black!50!black},
    citecolor={blue!50!black},
    urlcolor={blue!80!black}
}

\usepackage[hidelinks]{hyperref}
\usepackage{xcolor}
\hypersetup{
    colorlinks,
    linkcolor={red!50!black},
    citecolor={blue!50!black},
    urlcolor={blue!80!black}
}
\usepackage{graphicx}
\usepackage{caption}   
\usepackage{array,booktabs}    \usepackage{amsthm,amsmath,amsfonts,amssymb,latexsym,amscd,mathrsfs}
\usepackage{textcomp,float}
\usepackage[margin=25mm]{geometry}

\usepackage{tikz}
\usepackage{tikz-3dplot}
\usetikzlibrary{shapes,calc,positioning}

\usepackage{placeins}
\usepackage{hyperref}

\usepackage{geometry}
\definecolor{thistle}{rgb}{0.85, 0.75, 0.85}
\definecolor{blue-green}{rgb}{0.0, 0.87, 0.87}
\def\bF{\mathbb{F}}
\def\Fq{\mathbb{F}_q}
\def\GL{\mathrm{GL}}
\def\PG{\mathrm{PG}}
\def\PGL{\mathrm{PGL}}
\def\cP{\mathcal{P}}
\def\cL{\mathcal{L}}
\def\cV{\mathcal{V}}
\def\cZ{\mathcal{Z}}
\def\cC{\mathcal{C}}

\usepackage{graphicx}
\definecolor{amber(sae/ece)}{rgb}{1.0, 0.49, 0.0}
\definecolor{darkcyan}{rgb}{0.0, 0.55, 0.55}
\definecolor{darkseagreen}{rgb}{0.56, 0.74, 0.56}
\definecolor{salmon}{rgb}{1.0, 0.55, 0.41}

\usepackage{booktabs}
\usepackage{array}
\usepackage{paralist}
\usepackage{verbatim}
\usepackage{subfig}

\usepackage{fancyhdr}
\definecolor{carminepink}{rgb}{0.92, 0.3, 0.26}
	\definecolor{electricyellow}{rgb}{1.0, 1.0, 0.0}
	\definecolor{chromeyellow}{rgb}{1.0, 0.65, 0.0}
	\definecolor{babyblue}{rgb}{0.54, 0.81, 0.94}

\usepackage{sectsty}
\allsectionsfont{\sffamily\mdseries\upshape}

\def\cH{\mathcal{H}}

\usepackage[nottoc,notlof,notlot]{tocbibind}
\usepackage[titles,subfigure]{tocloft}

\usepackage{placeins}

\title{Solids in the space of the Veronese surface in even characteristic}

\author{Nour Alnajjarine\footnote{Sabanc\i~University, Istanbul, Turkey; \texttt{nour@sabanciuniv.edu}}, \; 
Michel Lavrauw\footnote{Sabanc\i~University, Istanbul, Turkey; \texttt{mlavrauw@sabanciuniv.edu}}, \; 
Tomasz Popiel\footnote{University of Auckland, Auckland, New Zealand; \texttt{tomasz.popiel@auckland.ac.nz}}}

\begin{document}

\maketitle

\begin{abstract}
We classify the orbits of solids in the projective space $\PG(5,q)$, $q$ even, under the setwise stabiliser $K \cong \PGL(3,q)$ of the Veronese surface. 
For each orbit, we provide an explicit representative $S$ and determine two combinatorial invariants: the {\em point-orbit distribution} and the {\em hyperplane-orbit distribution}.
These invariants characterise the orbits except in two specific cases (in which the orbits are distinguished by their {\em line-orbit distributions}). 
In addition, we determine the stabiliser of $S$ in $K$, thereby obtaining the size of each orbit. 
As a consequence, we obtain a proof of the classification of pencils of conics in $\PG(2,q)$, $q$ even, which to the best of our knowledge has been heretofore missing in the literature.
\end{abstract}

\section{Introduction}

Tensor product spaces are omnipresent in mathematics and science, with applications in areas including multilinear algebra, algebraic geometry, computational complexity theory, data analysis, finite geometry, quantum mechanics, and signal processing. 
Given a tensor product space $W=\otimes_{i=1}^m V_i$, where the $V_i$ are vector spaces over a fixed field $\bF$, a fundamental problem is to decompose a tensor $\tau \in W$ into a sum of `pure' tensors, namely to determine a decomposition of the form $\tau = \sum_{i=1}^n v_{1i} \otimes \dots \otimes v_{1m}$, and in particular to determine the {\em rank} of $\tau$, namely the minimal $n$ for which such a decomposition exists. 
For $m \geqslant 3$, the rank problem is very difficult in general \cite{hastad}, but often important: famously, the exponent of matrix multiplication is equal to the rank of a certain tensor in the case $m=3$. 

Another fundamental problem is to determine the equivalence classes of tensors in a given tensor product space, with respect to certain group actions. 
For example, the group $H=\GL(V_1) \times \dots \times \GL(V_m)$ acts on the set of pure tensors in $W$ via $(v_1 \otimes \dots \otimes v_m)^{(g_1,\dots,g_m)} = v_1^{g_1} \otimes \dots v_m^{g_m}$, and on all of $W$ by linearity. 
If some of the $V_i$ are isomorphic then one may also extend $H$ by a subgroup of the symmetric group $\operatorname{Sym}_m$ to obtain the full stabiliser $G \leqslant H \wr \operatorname{Sym}_m$ of the set of pure tensors in $W$. 
Naturally, one may then seek to classify the orbits of $G$ on $W$. 
For $m=2$, this is straightforward: elements of $V_1 \otimes V_2$ are matrices of size $\dim(V_1) \times \dim(V_2)$, and the action of $G$ preserves matrix rank, so there are $\min\{\dim(V_1),\dim(V_2)\}$ orbits. 
For $m \geqslant 3$, the problem becomes more difficult, and often depends on the field $\bF$ and/or the dimensions of the $V_i$. 
For example, a recent paper of the second author and Sheekey \cite{LaSh2015} provides a classification of the $G$-orbits in the case where $W = \bF^2 \otimes \bF^3 \otimes \bF^3$, for various fields $\bF$. 
It turns out that there are $15$ orbits when $\bF$ is algebraically closed, $17$ orbits when $\bF = \mathbb{R}$, and $18$ orbits when $\bF$ is a finite field. 
Of course, one may likewise seek to classify the {\em subspaces} of a given tensor product space. 
For example, by suitably contracting representatives of the aforementioned $G$-orbits of tensors in $W = \bF^2 \otimes \bF^3 \otimes \bF^3$, as per \cite[pp.~136--137]{LaSh2015}, one obtains a classification of the $2$-dimensional subspaces of $\bF^3 \otimes \bF^3$ under $\GL(3,q)^2 \wr \operatorname{Sym}_2$. 

Analogous questions arise in various natural subspaces of a tensor product space. 
For example, consider the space $W = S^n V$ of symmetric tensors in an $n$-fold tensor product $V^{\otimes n}$, namely the kernel of the action of $\operatorname{Sym}_n$ on $V^{\otimes n}$. 
Then the group $G = \GL(V)$ acts on $W$ via $(v_1 \otimes \dots \otimes v_m)^g = v_1^g \otimes \dots \otimes v_m^g$. 
In this setting, one finds interesting connections with algebraic geometry, because the pure tensors in $W$ correspond to the points of the Veronese variety in the projective space $\PG(W)$, and it can be advantageous to take this point of view if one aims to classify the $G$-orbits of subspaces of $W$. 
For instance, consider the case $n=2$ and $\dim(V)=3$. 
The pure tensors in $W$ correspond to the points of the Veronese surface $\cV(\bF)$ in $\PG(W) \cong \PG(5,\bF)$, and $G \cong \GL(3,\bF)$ induces a subgroup $K \cong \PGL(3,\bF)$ leaving $\cV(\bF)$ invariant. 
As explained in Section~\ref{pre}, the hyperplanes of $\PG(W) \cong \PG(5,\bF)$ are in one-to-one correspondence with the conics of $\PG(V) \cong \PG(2,\bF)$, and lower-dimensional subspaces correspond to {\em linear system of conics}: solids, planes and lines correspond to {\em pencils}, {\em nets} and {\em webs} of conics, namely $1$-, $2$- and $3$-dimensional systems, respectively. 
Classifications of linear systems of conics with respect to the natural action of $\PGL(3,\bF)$ on $\PG(2,\bF)$ are therefore equivalent to classifications of subspaces of $\PG(5,\bF)$ with respect to the induced action of the group $K$. 
Analogous problems may be studied in spaces of alternating and/or partially symmetric tensors, and there are natural connections between various settings. 
In particular, $k$-dimensional subspaces of $S^2 \bF^3$ correspond via a suitable contraction operation to tensors in $(S^2 \bF^3) \otimes \bF^k$. 
For further discussion and references, we refer the reader to the recent e-seminar \cite{michelSlides} of the second author, and to the closely related articles \cite{NourMichel,lines,nets,LaPoSh2021,LaSh2015,LaSh2017}. 

To motivate the specific results of the present paper, some history concerning linear systems of conics is in order. 
Linear systems of conics have been studied since at least 1906--1907, when Jordan~\cite{jordan1,jordan2} classified the pencils of conics in $\PG(2,\bF)$ for $\bF=\mathbb{C}$ and $\bF=\mathbb{R}$. 
The corresponding problem for finite fields was first studied in 1908, when Dickson~\cite{dickson} classified the pencils of conics in $\PG(2,\bF_q)$ for $q$ odd. 
Per the above discussion, these results yield a classification of solids in $\PG(5,\bF)$ for all of the mentioned fields~$\bF$. 
The obvious omission is the case in which $\bF$ is a finite field $\bF_q$ with $q$ even. 
Pencils of conics over finite fields of characteristic $2$ were studied in 1927 by Campbell~\cite{campbell}. 
However, as explained in \cite[Section~1.2]{lines}, Campbell's classification was incomplete and, in some respects, erroneous. 
Although a complete classification is alluded to in the literature, we have been unable to find it anywhere with an accompanying proof. 
In particular, the classification is stated in full by Hirschfeld as Theorem~7.31 of his book \cite{hirsch}, but the result is attributed to Campbell~\cite{campbell}, who neither stated nor proved a complete classification. 
The main aim of the present paper is to rectify this situation, by proving the following theorem:

\begin{theorem} \label{mainTheorem}
Let $q$ be an even prime power. 
There are exactly 15 orbits of pencils of conics in $\PG(2,q)$ with respect to the natural action of the projectivity group $\PGL(3,q)$. 
Equivalently, there are exactly 15 orbits of solids in $\PG(5,q)$ under the induced action of $\PGL(3,q) \leqslant \PGL(6,q)$ defined in Section~\ref{pre}. 
Representatives of these orbits are given in Table~\ref{mainTableNew}, the notation of which is also defined in Section~\ref{pre}. 
\end{theorem}

We also provide a significant amount of additional information about these orbits, summarised in Table~\ref{invariantsTable}. 
In particular, we calculate the stabiliser in $\PGL(3,q)$ of each representative, and thereby determine the size of each orbit. 
Additionally, we calculate two combinatorial invariants of each orbit, which can be used to distinguish between the orbits except in two specific cases. 
As explained in Section~\ref{pre}, there are exactly four $\PGL(3,q)$-orbits of points and four $\PGL(3,q)$-orbits of hyperplanes in $\PG(5,q)$. 
For each type of solid $S$, we calculate the {\em point-orbit distribution} of $S$, namely the number of points of each type belonging to $S$, and the {\em hyperplane-orbit distribution} of $S$, namely the number of hyperplanes of each type in which $S$ is contained. 
The latter can be interpreted in the setting of pencils of conics as counting the number of double lines, pairs of real lines, pairs of conjugate imaginary lines, and nonsingular conics contained in each type of pencil (see Section~\ref{pre}). 
Inspection of Table~\ref{invariantsTable} yields the following result. 

\begin{corollary} \label{corollaryDists}
Let $S$ and $S'$ be solids in $\PGL(5,q)$, $q$ even. 
Suppose that the point-orbit distributions of $S$ and $S'$ are equal, and that the hyperplane-orbit distributions of $S$ and $S'$ are equal. 
Then either (i) $S$ and $S'$ belong to the same $\PGL(3,q)$-orbit, (ii) $S$ and $S'$ belong to the union of the orbits $\Omega_{11}$ and $\Omega_{12}$, or (iii) $q=2$ and $S$ and $S'$ belong to the union of the orbits $\Omega_4$ and $\Omega_9$. 
\end{corollary}

In cases (ii) and (iii), one can decide whether $S$ and $S'$ belong to the same orbit by determining whether they intersect a certain orbit of {\em lines} in $\PG(5,q)$, as explained in Remark~\ref{o6Remark} and Section~\ref{F2}. 
The line orbits themselves were calculated by the second and third authors in \cite{lines}. 
We expect that the point- and hyperplane-orbit distributions given in Table~\ref{invariantsTable} will also have further applications, given that we have previously found data of this kind extremely useful in related work, both theoretical \cite{LaPoSh2021} and computational \cite{NourMichel}. 

The paper is structured as follows. 
In Section~\ref{pre} we make explicit the aforementioned connection between solids in $\PG(5,q)$ and pencils of conics in $\PG(2,q)$, using the Veronese map, and collect some preliminary lemmas. 
The proofs of Theorem~\ref{mainTheorem} and the associated data in Table~\ref{invariantsTable} are then given in Sections~\ref{section1}--\ref{F2}. 
Finally, we compare our results with the aforementioned work of Campbell~\cite{campbell} in Section~\ref{comparison1}.
Before proceeding, we note that our arguments intentionally exploit the connection between solids in $\PG(5,q)$ and pencils of conics in $\PG(2,q)$. 
By this we mean that we generally aim to use each point of view to its advantage. 
For instance, there seems to be no obvious way to calculate the point-orbit distribution of a solid by working directly with the associated pencil of conics.  
On the other hand, stabilisers are sometimes significantly easier to compute by working with pencils of conics, since we can appeal to well-known transitivity properties of the natural action of $\PGL(3,q)$ on $\PG(2,q)$, as opposed to having to work directly with matrix representations of the associated solids (see e.g. the proof of Lemma~\ref{s4}). 


\begin{table}[!ht]
\center
\small
\begin{tabular}{llll}
\toprule
Orbit & Representative
& Generating conics & Conditions \\
\midrule 
\vspace{2pt}
$\Omega_1$ & 
$\begin{bmatrix} x&y&z\\y&t&\cdot\\z&\cdot&t \end{bmatrix}$ & 
\makecell[l]{$(X_1+X_2)^2$ \\ $X_1X_2$}
& \\
\vspace{2pt}
$\Omega_2$ & 
$\begin{bmatrix} x&y&z\\y&t&\cdot\\z&\cdot&\cdot \end{bmatrix}$ & 
\makecell[l]{$X_2^2$ \\ $X_1X_2$}
& \\
\vspace{2pt}
$\Omega_3$ & 
$\begin{bmatrix} x&y&z\\y&\cdot&t\\z&t&\cdot \end{bmatrix}$ & 
\makecell[l]{$X_1^2$ \\ $X_2^2$}
& \\
\vspace{2pt}
$\Omega_4$ & 
$\begin{bmatrix} x&\cdot&y\\\cdot&z&\cdot\\y&\cdot&t \end{bmatrix}$ & 
\makecell[l]{$X_0X_1$ \\ $X_1X_2$}
& \\
\vspace{2pt}
$\Omega_5$ & 
$\begin{bmatrix} \cdot&x&y\\x&z&t\\y&t&x \end{bmatrix}$ & 
\makecell[l]{$X_0X_1+X_2^2$ \\ $X_0^2$}
& \\
\vspace{2pt}
$\Omega_6$ & 
$\begin{bmatrix} x&\cdot&y\\\cdot&z&t\\y&t&\cdot \end{bmatrix}$ & 
\makecell[l]{$X_0X_1+X_2^2$ \\ $X_2^2$}
& \\
\vspace{2pt}
$\Omega_7$ & 
$\begin{bmatrix} x&y&z\\ y&x+\gamma y&t\\z&t&y \end{bmatrix}$ & 
\makecell[l]{$X_0X_1+X_2^2$ \\ $(X_0+X_1+\gamma X_2)^2$}
& $\operatorname{Tr}(\gamma^{-1})=1$ \\
\vspace{2pt}
$\Omega_8$ & 
$\begin{bmatrix}x&y&z\\y&t&z\\z&z&y\end{bmatrix}$ & 
\makecell[l]{$X_0X_1+X_2^2$ \\ $(X_0+X_2)(X_1+X_2)$}
& \\
\vspace{2pt}
$\Omega_9$ & 
$\begin{bmatrix}x&x&y\\x&z&t\\y&t&t\end{bmatrix}$ & 
\makecell[l]{$X_0(X_0+X_1)$ \\ $X_2(X_1+X_2)$}
& \\
\vspace{2pt}
$\Omega_{10}$ & 
$\begin{bmatrix}x&y&z\\y& y+\gamma t& t\\ z& t&y\end{bmatrix}$ & 
\makecell[l]{$X_0X_1+X_2^2$ \\ $X_1(X_0+X_1+\gamma X_2)$}
& $\operatorname{Tr}(\gamma^{-1})=1$ \\
\vspace{2pt}
$\Omega_{11}$ & 
$\begin{bmatrix} x&y&z\\y&t&\cdot\\z&\cdot&y \end{bmatrix}$ & 
\makecell[l]{$X_0X_1+X_2^2$ \\ $X_1X_2$}
& \\
\vspace{2pt}
$\Omega_{12}$ & 
$\begin{bmatrix}x&y&z\\y&t&\gamma y+z\\z&\gamma y+z&y\end{bmatrix}$ & 
\makecell[l]{$X_0X_1+X_2^2$ \\ $X_2(X_0+X_1+\gamma X_2)$}
& $\operatorname{Tr}(\gamma^{-1})=1$ \\
\vspace{2pt}
$\Omega_{13}$ & 
$\begin{bmatrix} x&y&z\\y&\gamma x+y&t\\z&t&\gamma x+z \end{bmatrix}$ & 
\makecell[l]{$\gamma X_0^2 + X_0X_1 + X_1^2$ \\ $\gamma X_0^2 + X_0X_2 + X_2^2$}
& $\operatorname{Tr}(\gamma)=1$ \\
\vspace{2pt}
$\Omega_{14}$ & 
$\begin{bmatrix} x &y&\gamma x+y+\gamma t\\y&\gamma x+y&z\\\gamma x+y+\gamma t&z&t \end{bmatrix}$ & 
\makecell[l]{$X_1^2 + X_0X_2 + \gamma X_2^2$ \\ $\gamma X_0^2 + X_0X_1 + X_1^2$}
& $\operatorname{Tr}(\gamma)=1$ \\
\vspace{2pt}
$\Omega_{15}$ & 
$\begin{bmatrix}x&y&bz+cy\\y&z&t\\bz+cy&t&y\end{bmatrix}$ & 
\makecell[l]{$X_0X_1+X_2^2$ \\ $X_0X_2 + bX_1^2 + cX_2^2$}
& $b\lambda^3 + c\lambda + 1$ irreducible over $\bF_q$ \\
\bottomrule
\end{tabular}
\caption{The $\PGL(3,q)$-orbits of solids in $\PG(5,q)$ and pencils of conics in $\PG(2,q)$, $q$ even. Notation is as defined in Section~\ref{pre}; in particular, notation in the second column is as in \eqref{egSolid}.}
\label{mainTableNew}
\end{table}


\begin{table}
\center
\small
\begin{tabular}{lllll}
\toprule
Orbit & Point OD & Hyperplane OD & Stabiliser & Orbit size \\
\midrule
$\Omega_1$ & $[1,q+1,2q^2-1,q^3-q^2]$ & $[1,q/2,q/2,0]$ & $E_q^2 : (E_q \times C_{q-1})$ & $(q^3-1)(q+1)$ \\
$\Omega_2$ & $[q+1,q+1,2q^2-q-1,q^3-q^2]$ & $[1,q,0,0]$ & $E_q^{1+2} : C_{q-1}^2$ & $(q^2+q+1)(q+1)$ \\
$\Omega_3$ & $[1,q^2+q+1,q^2-1,q^3-q^2]$ & $[q+1,0,0,0]$ & $E_q^2 : \GL(2,q)$ & $q^2+q+1$ \\
$\Omega_4$ & $[q+2,1,2q^2-2,q^3-q^2]$ & $[0,q+1,0,0]$ & $\GL(2,q)$ & $q^2(q^2+q+1)$ \\
$\Omega_5$ & $[1,q+1,q^2-1,q^3]$ & $[1,0,0,q]$ & $E_q^2 : C_{q-1}$ & $q(q^3-1)(q+1)$ \\
$\Omega_6$ & $[2,q+1,q^2+q-2,q^3-q]$ & $[1,1,0,q-1]$ & $C_{q-1}^2 : C_2$ & $\tfrac{1}{2}q^3(q^2+q+1)(q+1)$ \\
$\Omega_7$ & $[0,q+1,q^2+q,q^3-q]$ & $[1,0,1,q-1]$ & $D_{2(q+1)} \times C_{q-1}$ & $\tfrac{1}{2}q^3(q^3-1)$ \\
$\Omega_8$ & $[3,1,q^2+2q-3,q^3-q]$ & $[0,2,0,q-1]$ & $C_{q-1} \times C_2$ & $\tfrac{1}{2}q^3(q^3-1)(q+1)$ \\
$\Omega_9$ & $[4,1,q^2+3q-4,q^3-2q]$ & $[0,3,0,q-2]$ & $\operatorname{Sym}_4$ & $\tfrac{1}{24}q^3(q^3-1)(q^2-1)$ \\
$\Omega_{10}$ & $[1,1,q^2+2q-1,q^3-q]$ & $[0,1,1,q-1]$ & $C_{q-1} \times C_2$ & $\tfrac{1}{2}q^3(q^3-1)(q+1)$ \\
$\Omega_{11}$ & $[2,1,q^2+q-2,q^3]$ & $[0,1,0,q]$ & $E_q : C_{q-1}$ & $q^2(q^3-1)(q+1)$ \\
$\Omega_{12}$ & $[2,1,q^2+q-2,q^3]$ & $[0,1,0,q]$ & $C_2^2$ & $\tfrac{1}{4}q^3(q^3-1)(q^2-1)$ \\
$\Omega_{13}$ & $[0,1,q^2+3q,q^3-2q]$ & $[0,1,2,q-2]$ & $C_2^2 : C_2$ & $\tfrac{1}{8}q^3(q^3-1)(q^2-1)$ \\
$\Omega_{14}$ & $[0,1,q^2+q,q^3]$ & $[0,0,1,q]$ & $C_4$ & $\tfrac{1}{4}q^3(q^3-1)(q^2-1)$ \\
$\Omega_{15}$ & $[1,1,q^2-1,q^3+q]$ & $[0,0,0,q+1]$ & $C_3$ & $\tfrac{1}{3}q^3(q^3-1)(q^2-1)$ \\
\bottomrule
\end{tabular}
\caption{Invariants of $\PGL(3,q)$-orbits of solids in $\PG(5,q)$, $q$ even. 
Point- and hyperplane-orbit distributions (OD) are defined in  Section~\ref{pre}. 
In the fourth column, $E_q$ denotes an elementary abelian group of order~$q$, and $E_q^{1+2}$ is a group with centre $Z \cong E_q$ such that $E_q^{1+2}/Z \cong E_q^2$ (e.g. the group of upper-unitriangular $3 \times 3$ matrices over $\bF_q$); $C_k$, $D_k$ and $\operatorname{Sym}_k$ are cyclic, dihedral and symmetric groups, respectively.}
\label{invariantsTable}
\end{table}


\section{Preliminaries}\label{pre}

In this section, we review some definitions and theory needed in our proofs, most of which can be found in \cite{harris,lines}. 
We also refer the reader to \cite{Havlicek} for an overview of the interesting properties of the Veronese surface over finite fields. 
The {\it Veronese surface} $\cV(\bF_q)$ is a 2-dimensional algebraic variety in $\PG(5,q)$ which can be defined as the image of the Veronese map
\[
\nu: \PG(2,q)\rightarrow \PG(5,q) \quad \text{given by} \quad
(u_0,u_1,u_2) \mapsto (u_0^2,u_0u_1,u_0u_2,u_1^2,u_1u_2, u_2^2).
\]
A point $P=(y_0,y_1,y_2,y_3,y_4,y_5)$ of $\PG(5,q)$ can also be represented by a symmetric $3 \times 3$ matrix
\[
M_P=\begin{bmatrix} y_0&y_1&y_2\\
y_1&y_3&y_4\\
y_2&y_4&y_5  \end{bmatrix}.
\]
This representation can be extended to any subspace of $\PG(5,q)$. 
For example, the solid spanned by the first four points of the standard frame of $\PG(5,q)$ is represented by 
\begin{equation} \label{egSolid}
\begin{bmatrix} x&y&z\\y&t&\cdot\\z&\cdot&\cdot \end{bmatrix}:=\left\{
\begin{bmatrix} x&y&z\\y&t&0\\z&0&0\end{bmatrix}:(x,y,z,t)\in \mathbb{F}_q^4 \setminus \{(0,0,0,0)\}
\right\},
\end{equation}
where the notation on the left is introduced for convenience (that is, $\cdot$ represents $0$, and the $4$-tuple $(x,y,z,t)$ is understood to range over all non-zero elements of $\mathbb{F}_q^4$). 
This notation is used throughout the paper.  

The {\it rank} of a point $P$ of $\PG(5,q)$ is defined to be the rank of the matrix $M_P$. 
The points of rank $1$ are (therefore precisely) those belonging to $\mathcal{V}(\Fq)$. 
Points of $\PG(5,q)$ of rank at most $2$ are the points of the {\em secant variety} of $\mathcal{V}(\Fq)$, which we denote by $\mathcal{V}(\Fq)^2$.
Lines of $\PG(2,q)$ are mapped by $\nu$ to conics lying on $\cV(\bF_q)$, and if $q>2$ then each conic lying on $\cV(\bF_q)$ is the image of a line of $\PG(2,q)$. 
Consequently, each two points $P$, $Q$ of $\mathcal{V}(\Fq)$ lie on a unique conic contained in $\cV(\bF_q)$, given by $\mathcal{C}_{P,Q}:= \nu(\langle\nu^{-1}(P),\nu^{-1}(Q)\rangle)$. 
If $q$ is even, then all tangent lines to a nonsingular conic $\mathcal{C}$ are concurrent, meeting at a point called the {\em nucleus} of $\cC$. 
The set of all nuclei of conics contained in $\mathcal{V}(\Fq)$ coincides with the set of points of a plane in $\PG(5,q)$ known as the {\em nucleus plane} of $\cV(\bF_q)$. 
In the above representation, points contained in the nucleus plane correspond to symmetric $3\times 3$ matrices with zeros on the main diagonal. 
Each rank-$2$ point $R$ of $\PG(5,q)$ defines a unique conic $\mathcal{C}_{R}$ lying on $\cV(\bF_q)$. 
If $R$ lies on the secant $\langle P,Q \rangle$ with $P,Q \in \mathcal{V}(\Fq)$ then $\cC_R=\cC_{P,Q}$. 
If $q$ is even and $R$ is contained in the nucleus plane then $R$ is the nucleus of $\cC_R$.

Given a solid $S$ of $\PG(5,q)$, setting the determinant of the matrix representing $S$ to zero defines a cubic surface, which we denote by $\Psi(S)$. 
For example, for the solid given in \eqref{egSolid}, $\Psi(S)$ is the cubic surface comprising points as in \eqref{egSolid} with $zt=0$. 
In particular, we see that $S$ has {\em rank distribution} $[q+1,2q^2,q^3-q^2]$, meaning that it contains $q+1$ points of rank~$1$, $2q^2$ points of rank~$2$, and $q^3-q^2$ points of rank~$3$. 
(The points of rank~$1$ comprise the nonsingular conic given by $z=0$ and $xt=y^2$.)
In general, we define the rank distribution of a subspace $U$ of $\PG(5,q)$ to be the $3$-tuple $[r_1,r_2,r_3]$, where $r_i$ is the number of rank~$i$ points in $U$. 
The rank distribution is related to a particular case of what we call an {\it orbit distribution}. 

\begin{definition}\label{orbitdist}
\textnormal{
Let $G \leqslant {\mathrm{P\Gamma L}}(n+1,q)$ and let $U_1,U_2,\dots,U_m$ denote (a chosen ordering of) the distinct $G$-orbits of $r$-spaces in $\PG(n,q)$. 
The \textit{$r$-space $G$-orbit distribution} of a subspace $U$ of $\PG(n,q)$ is the list 
\[
\operatorname{OD}_{G,r}(W):=[u_1,u_2,\ldots,u_m],
\]
where $u_i$ is the number elements of $U_i$ incident with $U$.
}
\end{definition}

We are interested in the action on subspaces of $\PG(5,q)$ of the group $K\leqslant \PGL(6,q)$ defined as the lift of $\PGL(3,q)$ through the Veronese map. 
Explicitly, if $\phi_A\in \PGL(3,q)$ is represented by the matrix $A \in \GL(3,q)$ then we define the corresponding projectivity $\alpha(\phi_A)\in \PGL(6,q)$ through its action on the points of $\PG(5,q)$ by 
\[
\alpha(\phi_A): P\mapsto Q \quad \text{where} \quad M_Q=AM_PA^T.
\] 
Then $K:=\alpha(\PGL(3,q))$ is isomorphic to $\PGL(3,q)$ and leaves $\cV(\bF_q)$ invariant. 

The rank distribution of a subspace of $\PG(5,q)$ is related to its $0$-space $K$-orbit distribution as follows. 
There are four $K$-orbits of $0$-spaces, i.e. points, in $\PG(5,q)$: the orbit $\cV(\bF_q)$ of rank-$1$ points, which has size $q^2+q+1$, the orbit of rank-$3$ points, which has size $q^5-q^2$, and two orbits of rank-$2$ points. 
For $q$ even, the orbits of rank-$2$ points comprise the $q^2+q+1$ points of the nucleus plane $\pi_n$, and the $(q^2-1)(q^2+q+1)$ points contained in conic planes but not in $\pi_n$. 
Therefore, the orbit distribution $\operatorname{OD}_{K,0}(U)$ of a subspace $U$ of $\PG(5,q)$, $q$ even, is the $4$-tuple $[r_1,r_{2n},r_{2s},r_3]$, where $r_i$, $i\in \{1,3\}$, is the number of rank-$i$ points in $U$, $r_{2n}$ is the number of rank-$2$ points in $U \cap \pi_n$, and $r_{2s}$ is the number of rank-$2$ points in $U \setminus \pi_n$. 

For brevity, we also call $\operatorname{OD}_{K,i}(U)$ with $i=0$ the {\em point-orbit distribution} of a subspace $U$ of $\PG(5,q)$. 
Similarly, we obtain the {\it line-}, {\it plane-}, {\it solid-}, and {\it hyperplane-orbit distributions} of $U$ for $i=1,2,3,4$ respectively. 
These data serve as useful invariants for studying $K$-orbits of subspaces of $\PG(5,q)$. 
For example, if $q$ is odd and $U$ is a plane containing at least one point of $\cV(\bF_q)$, then the line-orbit distribution of $U$ completely determines its $K$-orbit \cite{LaPoSh2021}.
The line orbits themselves were determined (for all $q$) in \cite{lines}. 

\begin{theorem}\label{lines}
There are 15 $K$-orbits of lines in $\PG(5,q)$, as listed in \cite[Table~2]{lines}.
\end{theorem}

Hyperplanes of $\PG(5,q)$ correspond to conics of $\PG(2,q)$ through the Veronese map $\nu$. We make this correspondence explicit via the following map $\delta$ between conics of $\PG(2,q)$ and hyperplanes of $\PG(5,q)$:
\[
\delta: \cZ\Big (\sum_{0\leqslant i \leqslant j \leqslant 2}a_{ij}X_iX_j\Big)\mapsto \mathcal{Z}(a_{00}Y_0+a_{01}Y_1+a_{02}Y_2+a_{11}Y_3+a_{12}Y_4+a_{22}Y_5).
\]
Here, and throughout the paper, the homogeneous coordinates in the domain $\PG(2,q)$ of $\nu$ are denoted by $(X_0,X_1,X_2)$, the homogeneous coordinates in $\PG(5,q)$ are denoted by $(Y_0,\ldots,Y_5)$, and $\cZ(f)$ denotes the zero locus of a form $f$. 
Note that a point $P$ in $\PG(2,q)$ lies on a (given) conic $\cC$ if and only if $\nu(P)$ lies in the hyperplane $\delta(\cC)$. 
The definition of $\delta$ extends to a set $\mathcal{S}$ of conics in the obvious way:
\[
\delta(\mathcal{S}) = \cap_{\cC \in \mathcal{S}} \delta(\cC).
\]

Up to projective equivalence, there is a unique nonsingular conic in $\PG(2,q)$, and three classes of singular conics, namely (i) double lines, (ii) pairs of real lines, and (iii) pairs of (conjugate) imaginary lines. 
We denote the corresponding $K$-orbits of hyperplanes (obtained via $\delta$) as follows: $\cH_{1}$, $\cH_{2r}$ and $\cH_{2i}$ denote the $K$-orbits of hyperplanes corresponding to the $\PGL(3,q)$-orbits of singular conics of types (i), (ii) and (iii) respectively, and $\cH_3$ denotes the $K$-orbit of hyperplanes corresponding to the $\PGL(3,q)$-orbit of nonsingular conics. 
The following criterion determines when a conic is nonsingular. 

\begin{lemma}\label{pencils}
A conic $\cZ(f)$ defined by a quadratic form $f=\sum_{0\leqslant i \leqslant j \leqslant 2}a_{ij}X_iX_j$ in $\PG(2,q)$, $q$ even, is absolutely irreducible (or, equivalently, nonsingular) if and only if $a_{00}a_{12}^2+a_{11}a_{02}^2+a_{22}a_{01}^2+a_{01}a_{02}a_{12}\neq 0$. 
\end{lemma}

Subspaces of $\PG(5,q)$ correspond to linear systems of conics in $\PG(2,q)$ via $\nu$: lines correspond to {\it webs}, planes to {\it nets}, and solids to {\it pencils} of conics in $\PG(2,q)$. 
In particular, the classification of webs of conics over finite fields is equivalent to Theorem~\ref{lines}. 
The {\it base} (or set of {\it base points}) of a linear system of conics is the intersection of the conics in the system. 
We make the following observation. 

\begin{lemma} \label{baseLemma}
Let $Q$ be a point and $\cP$ a pencil of conics in $\PG(2,q)$. 
Then $Q$ is a base point of $\cP$ if and only if $\nu(Q)$ lies in the solid $S = \delta(\cP)$ of $\PG(5,q)$.
\end{lemma}

In other words, the points of rank~$1$ in $S$ are precisely the images under the Veronese map  of the base points of $\cP$.
The following lemma about quadratic polynomials over finite fields of characteristic $2$ is well known \cite{roots}. 
Here, and throughout the paper, $\operatorname{Tr}$ denotes the trace map from $\bF_{2^n}$ to $\bF_2$. 

\begin{lemma}\label{quadratic}
The polynomial $\alpha X^2 + \beta X + \gamma \in \mathbb{F}_{2^n}[X]$ with $\alpha \neq 0$ has exactly one root in $\mathbb{F}_{2^n}$ if and only if $\beta = 0$, two distinct roots in $\mathbb{F}_{2^n}$ if and only if $\beta \neq 0$ and $\operatorname{Tr}(\frac{\alpha \gamma}{\beta^2})=0$, and no roots in $\mathbb{F}_{2^n}$ otherwise.
\end{lemma}

We conclude this section with a lemma concerning the hyperplane-orbit distribution $\operatorname{OD}_{K,4}(S)=[a_1,a_{2r},a_{2i},a_3]$ of a solid in $\PG(5,q)$, $q$ even. 
As per above, here $a_j$ denotes the number of hyperplanes of type $\cH_j$ intersecting $U$ for each of the symbols $j \in \{1,2r,2i,3\}$. 

\begin{lemma}\label{prop}
Let $S$ be a solid of $\PG(5,q)$, where $q=2^h$ with $h>1$, and let $b$ denote the number of points of $S$ contained in $\cV(\bF_q)$. 
Then the hyperplane-orbit distribution $\operatorname{OD}_{K,4}(S)=[a_1,a_{2r},a_{2i},a_3]$ of $S$ satisfies \textnormal{(i)} $a_1+2 a_{2r}+a_3=q+b$ and \textnormal{(ii)} $a_{2r}-a_{2i}+1=b$. 
\end{lemma}

\begin{proof}
First note that $(q+1)a_1+(2q+1)a_{2r}+a_{2i}+(q+1)a_3-bq=q^2+q+1$. 
This follows from the fact that each point on $\cV(\bF_q)$ either lies in $S$ and belongs to $q+1$ hyperplanes through $S$, or belongs to exactly one hyperplane of $\PG(5,q)$ through $S$, and the fact that the hyperplanes in the orbits $\cH_1$, $\cH_{2r}$, $\cH_{2i}$, $\cH_3$ intersect $\cV(\bF_q)$ in $q+1$, $2q+1$, $1$, $q+1$ points respectively. 
Now use the fact that $a_1+a_{2r}+a_{2i}+a_3=q+1$ and divide by $q$ to get (i). 
Substitution of $a_1+a_{2r}+a_3$ by $q+1-a_{2i}$ gives (ii).
\end{proof}

The following three sections constitute the proof of Theorem~\ref{mainTheorem} and of the data in Table~\ref{invariantsTable} for $q > 2$. 
The case $q=2$ requires special treatment, and is handled in Section~\ref{F2}.

\section{Solids not contained in any hyperplane of type $\mathcal{H}_3$}\label{section1}

We begin by classifying the $K$-orbits of solids that are not contained in any hyperplane of type $\cH_3$, namely, those for which the corresponding pencil of conics contains no nonsingular conics. 
It is straightforward to list the possible configurations of pairs of conics that can occur. However, since we are interested in $K$-orbits of solids, i.e. pencils of conics as opposed to pairs of conics, we need to understand when two different types of pairs of conics give rise to the same pencil up to projective equivalence.

Here, and in subsequent sections, homogeneous coordinates in a solid $S$ of $\PG(5,q)$ are generally denoted by $(X,Y,Z,T)$, where the solid is represented as in \eqref{egSolid}. 
The pencil of conics in $\PG(2,q)$ corresponding to $S$ is denoted by $\cP(S)$, and the cubic surface obtained as the intersection of $S$ with the secant variety $\cV(\bF_q)^2$ of $\cV(\bF_q)$ is denoted by $\Psi(S)$.
As before, the homogeneous coordinates in the domain $\PG(2,q)$ of the Veronese map $\nu$ are denoted by $(X_0,X_1,X_2)$, and those in $\PG(5,q)$ are denoted by $(Y_0,\ldots,Y_5)$.

\subsection{Solids contained in a hyperplane of type $\mathcal{H}_1$}

We first treat the $K$-orbits of solids $S$ corresponding to pencils $\cP(S)$ that contain at least one double line, namely those whose hyperplane-orbit distribution $\operatorname{OD}_{K,4}(S)=[a_1,a_{2r},a_{2i},a_3]$ has $a_1 > 0$. 
If $\cP(S)$ contains exactly one double line (i.e. $a_1=1$) and no pair of (distinct) real lines ($a_{2r}=0$), then the orbit of $S$ will arise later in our analysis, since any such pencil contains a nonsingular conic, by the following lemma.

\begin{lemma}
If $\operatorname{OD}_{K,4}(S)=[1,0,a_{2i},a_3]$ with $a_{2i}>0$, then $a_3>0$.
\end{lemma}

\begin{proof}
Putting $a_1=1$ and $a_{2r}=0$ into Lemma~\ref{prop}(ii) gives $b=1-a_{2i}$, which implies that $a_{2i} \leqslant 1$ since $b \geqslant 0$. 
Therefore, $a_{2i}=1$ and so $a_3 = (q+1)-2 > 0$.
\end{proof}

We may therefore assume that if $\cP(S)$ contains exactly one double line, say $\cL_1^2$, then it contains at least one pair of distinct real lines, say $\cL_2\cL_3$. 
We then have the following possibilities: (i) the three lines are distinct and concurrent, (ii) $\cL_1$ coincides with one of $\cL_2$ or $\cL_3$, or (iii) the three lines are distinct and not concurrent.
Since $\PGL(3,q)$ acts transitively on each of these configurations of lines, two solids corresponding to the same configuration belong to the same $K$-orbit. 
In case (iii), $\cP(S)$ has exactly two base points, so the following lemma implies, together with Lemma~\ref{baseLemma}, that $S$ is contained in a hyperplane of type $\cH_3$.

\begin{lemma}
If $\operatorname{OD}_{K,4}(S)=[1,a_{2r},a_{2i},a_3]$ and $S$ meets $\cV(\bF_q)$ in two points, then $a_{2i}=0$ and $a_3>0$.
\end{lemma}

\begin{proof}
Since a hyperplane of type $\cH_{2i}$ meets $\cV(\bF_q)$ in one point, it follows that $a_{2i}=0$. 
Putting $a_1=1$ and $b=2$ into Lemma~\ref{prop}(i) gives $a_3 = (q+1) - 2a_{2r}$. 
Since $q$ is even, $(q+1) - 2a_{2r}$ is odd, so $a_3 \geqslant 1$.
\end{proof}

If follows that we are left with at most two $K$-orbits, corresponding to the cases (i) and (ii). 
We label these orbits as $\Omega_1$ and $\Omega_2$ respectively and choose representatives for them as 
\begin{equation} \label{S1S2reps}
\Omega_1:\begin{bmatrix} x&y&z\\y&t&\cdot\\z&\cdot&t \end{bmatrix}, 
\quad
\Omega_2:\begin{bmatrix} x&y&z\\y&t&\cdot\\z&\cdot&\cdot \end{bmatrix},
\end{equation}
obtained by taking $\cL_2 = \cZ(X_1)$ and $\cL_3 = \cZ(X_2)$ in both cases, $\cL_1 = \cZ((X_1+X_2)^2)$ for $\Omega_1$ and $\cL_1 = \cZ(X_2^2)$ for $\Omega_2$. 
We now calculate the point-orbit distributions, hyperplane-orbit distributions, and stabilisers of the solids in these $K$-orbits. 
We may use the representatives given in \eqref{S1S2reps} for these calculations, since all of the aforementioned data are $K$-invariant. 
We begin with the hyperplane-orbit distributions, verifying in particular the desired condition that each solid lies in a unique hyperplane of type $\cH_1$, and that the orbits $\Omega_1$ and $\Omega_2$ are indeed distinct (since their hyperplane-orbit distributions are distinct). 

\begin{lemma} \label{12dists}
The hyperplane-orbit distribution of a solid of type $\Omega_1$ is $[1,q/2,q/2,0]$. 
The hyperplane-orbit distribution of a solid of type $\Omega_2$ is $[1,q,0,0]$. 
In particular, $\Omega_1 \neq \Omega_2$. 
\end{lemma}

\begin{proof}
Let $S_i$ denote the representative of $\Omega_i$ defined in \eqref{S1S2reps}, for $i \in \{1,2\}$. 
Lemma~\ref{pencils} implies that each of the pencils $\cP(S_i)$ does indeed contain a unique double line (namely $\cL_1^2$) and no nonsingular conics. 
Hence, in the notation of Lemma~\ref{prop}, the hyperplane-orbit distribution of $S_i$ has the form $[1,a_{2r},a_{2i},0]$ in both cases, i.e. $a_1=1$ and $a_3=0$. 
The pencil $\cP(S_1)$ has a unique base point (the unique point of concurrency of the three lines $\cL_1$, $\cL_2$ and $\cL_3$), so putting $b=1$ into Lemma~\ref{prop} yields $a_{2r}=a_{2i}=q/2$.
On the other hand, $\cP(S_2)$ has $q+1$ base points (those on the line $\cL_1$), so $a_{2r}=q$ and $a_{2i}=0$.
\end{proof}

\begin{lemma}
The point-orbit distribution of a solid of type $\Omega_1$ is $[1,q+1,2q^2-1,q^3-q^2]$. 
The point-orbit distribution of a solid of type $\Omega_2$ is $[q+1,q+1,2q^2-q-1,q^3-q^2]$. 
\end{lemma}

\begin{proof}
Consider again the representatives $S_1$ and $S_2$ in \eqref{S1S2reps}. 
Points of rank at most $2$ in $S_1$ correspond to points on the cubic surface $\Psi(S_1)=\cZ(XT^2+Y^2T+Z^2T)$. 
There are $2q^2+q+1$ such points, exactly one of which has rank $1$, namely the point with homogeneous coordinates $(X,Y,Z,T) = (1,0,0,0)$, which is the image under $\nu$ of the unique base point $(X_0,X_1,X_2) = (1,0,0)$ of the pencil $\cP(S_1)$ (cf. Lemma~\ref{baseLemma}). 
Hence, the rank distribution of $S_1$ is $[1,2q^2+q,q^3-q^2]$.
The points of $S_1$ contained in the nucleus plane are those on the line $\cZ(X,T)$, so the point-orbit distribution of $S_1$ is $[1,q+1,2q^2-1,q^3-q^2]$. 
The cubic surface $\Psi(S_2)$ is $\cZ(Z^2T)$, which contains $2q^2+q+1$ points, being the union of two planes meeting in a line. 
It intersects $\cV(\bF_q)$ in the conic $\cZ(Z, XT+Y^2)$, and the nucleus plane in the line $\cZ(X,T)$.
\end{proof}

In the following lemma (and elsewhere in the paper), $E_q$ denotes an elementary abelian group of order~$q$, and $E_q^{1+2}$ is a group with centre $Z \cong E_q$ such that $E_q^{1+2}/Z \cong E_q^2$. 

\begin{lemma}
If $S_1 \in \Omega_1$ then $K_{S_1} \cong E_q^2 : (E_q \times C_{q-1})$. 
If $S_2 \in \Omega_2$ then $K_{S_2} \cong E_q^{1+2} : C_{q-1}^2$. 
\end{lemma}

\begin{proof}
If $S_1$ is the representative of $\Omega_1$ given in \eqref{S1S2reps} then $K_{S_1} \leqslant K_P$, where $P=(1,0,0,0)$ is the unique point of rank~$1$ in $S_1$. 
Notice that $K_P$ is equal to the stabiliser of the plane $\pi = \mathcal{Z}(T)$, because $\pi$ is the tangent plane to $\mathcal{V}(\mathbb{F}_q)$ at $P$. 
An element of $K_P$ therefore fixes $S_1$ if and only if it maps the point $Q=(0,0,0,1)$ into $S_1$, since $S_1 = \langle \pi,Q \rangle$. 
Elements of $K_P \cong E_q^2 : \GL(2,q)$ are represented by matrices $g=(g_{ij}) \in \GL(3,q)$ with $g_{21}=g_{31}=0$. 
The subgroup $H \cong \GL(2,q)$ of $K_P$ obtained by setting $g_{12}=g_{13}=0$ fixes the conic $\mathcal{Z}(Y_3^2+Y_4Y_5)$ in the plane $\pi' = \cZ(Y_0,Y_1,Y_2)$. 
Since $Q$ is a point external to this conic and distinct from its nucleus, it follows by considering the quotient space of $\pi'$ that $K_{S_1} \cong E_q^2 : (E_q \times C_{q-1})$.
The solid $S_2 \in \Omega_2$ given in \eqref{S1S2reps} meets $\mathcal{V}(\mathbb{F}_q)$ in a conic which spans the plane $\pi : \cZ(Z)$, so $K_{S_2} \leqslant K_\pi$. 
An element of $K$ represented by a matrix $(g_{ij}) \in \GL(3,q)$ belongs to $K_\pi$ if and only $g_{31}=g_{32}=0$. 
It fixes $S_2$ if and only if it also fixes the line $\cZ(X,T)$ in which $S_2$ intersects the nucleus plane. 
This occurs if and only if $g_{21}$ is also $0$. 
Upon factoring out scalars we therefore obtain $K_{S_2} \cong E_q^{1+2} : C_{q-1}^2$. 
\end{proof}

If $\cP(S)$ contains more than one double line, then it is a pencil of lines. 
There is one $K$-orbit of such solids, which we call $\Omega_3$. 
Generating $\cP(S)$ by the double lines $\cZ(X_1^2)$ and $\cZ(X_2^2)$ gives the representative
\[
\Omega_3:\begin{bmatrix} x&y&z\\y&\cdot&t\\z&t&\cdot \end{bmatrix}.
\]

\begin{lemma} \label{3data}
A solid $S_3 \in \Omega_3$ has point-orbit distribution $[1,q^2+q+1,q^2-1,q^3-q^2]$, hyperplane-orbit distribution $[q+1,0,0,0]$, and stabiliser $K_{S_3} \cong E_q^2 : \GL(2,q)$. 
In particular, $\Omega_3 \not \in \{\Omega_1,\Omega_2\}$. 
\end{lemma} 
 
\begin{proof}
Let $S_3$ denote the above representative of $\Omega_3$. 
Since all conics in the pencil $\cP(S_3)$ are double lines, the hyperplane-orbit distribution of $S_3$ is $[q+1,0,0,0]$. 
This implies that $\Omega_3 \not \in \{\Omega_1,\Omega_2\}$ (upon comparing with Lemma~\ref{12dists}). 
The cubic surface $\Psi(S_3)$ is the union of the nucleus plane $\cZ(X)$ and the plane $\cZ(T)$. 
It contains exactly one point of rank~$1$, namely the point $P=(1,0,0,0)$. 
Therefore, we obtain the asserted point-orbit distribution. 
The stabiliser is immediate from the hyperplane-orbit distribution. 
\end{proof}

\subsection{Solids not contained in a hyperplane of type $\mathcal{H}_1$}

Next we classify the solids contained neither in hyperplanes of type $\mathcal{H}_3$, nor in hyperplanes of type of $\mathcal{H}_1$.
Let $S$ be such a solid, namely one with with $\operatorname{OD}_{K,4}(S)=[0,a_{2r},a_{2i},0]$. 
Since we are assuming that $q>2$, it follows from Lemma~\ref{prop}(i) that $a_{2r} \geqslant 2$.
Hence, there exist two pairs $\cL_1\cL_2$ and $\cL_3\cL_4$ of distinct real lines generating $\cP(S)$.
There a number of possible configurations of the lines $\cL_1,\ldots,\cL_4$, but it turns out that only one of these gives a $K$-orbit with the assumed hyperplane-orbit distribution.

\begin{lemma}
There is a unique $K$-orbit of solids with hyperplane-orbit distribution $[0,a_{2r},a_{2i},0]$.
\end{lemma}

\begin{proof}
If the four lines $\cL_1,\ldots,\cL_4$ are concurrent, then $S\in \Omega_1$. 
If $\cP(S)$ has one of the lines as its base, then $S\in\Omega_2$ since $\cP(S)$ then also contains that base as a double line.
If the two pairs $\cL_1\cL_2$ and $\cL_3\cL_4$ meet in either three or four points, then $\cP(S)$ contains at least one nonsingular conic (and so $a_3 \neq 0$): this can be verified by a direct computation, and also follows from the treatment of the orbits $\Omega_8$ and $\Omega_9$ in Section~\ref{kk=22}.
The only remaining possibility is that the two pairs share a line and do not meet in the same point, in which case the base of $\cP(S)$ is an antiflag, consisting of the shared line and one extra point. 
Since $\PGL(3,q)$ acts transitively on antiflags, there is one such $K$-orbit of solids.
\end{proof}

The $K$-orbit of solids arising as above is denoted $\Omega_4$. 
Taking $\cP(S)$ generated by the pairs of real lines $\cZ(X_0X_1)$ and $\cZ(X_1X_2)$ gives the representative
\[
\Omega_4: \begin{bmatrix} x&\cdot&y\\\cdot&z&\cdot\\y&\cdot&t \end{bmatrix}.
\]

\begin{lemma}\label{s4}
A solid $S_4 \in \Omega_4$ has point-orbit distribution $[q+2,1,2q^2-2,q^3-q^2]$, hyperplane-orbit distribution $[0,q+1,0,0]$, and stabiliser $K_{S_4} \cong \GL(2,q)$. 
In particular, $\Omega_4 \not \in \{\Omega_1,\Omega_2,\Omega_3\}$. 
\end{lemma}

\begin{proof}
Let $S_4$ be the solid defined above. 
Every conic in the pencil $\cP(S_4)$ has the form $\cZ(X_1(\lambda X_0 + \mu X_2))$ for some $\lambda$, $\mu$, namely a pair of real lines, so the hyperplane-orbit distribution is $[0,q+1,0,0]$, and this implies that $\Omega_4 \not \in \{\Omega_1,\Omega_2,\Omega_3\}$. 
The cubic surface $\Psi(S_4)=\cZ(Z(XT+Y^2))$ is the union of a plane and a quadratic cone with vertex $P=(0,0,1,0)$, meeting in a conic $\mathcal C = \cZ(Y_0Y_5+Y_2^2)$. 
It intersects $S_4$ in $P \cup \cC$, so $S_4$ contains $q+2$ points of rank~$1$. 
The nucleus of $\cC$ is the unique point of $S_4$ in the nucleus plane.
The pencil $\cP(S_4)$ is fixed by an element of $\PGL(3,q)$ if and only if the antiflag comprising its base is fixed, so $K_{S_4}$ is isomorphic to the stabiliser of an antiflag, i.e. $K_{S_4} \cong \GL(2,q)$.
\end{proof}

This completes the classification of solids contained in no hyperplane of type $\cH_3$, or equivalently, of pencils of conics containing no nonsingular conics. 
We make the following observation for reference. 

\begin{corollary}
There is no pencil of conics in $\PG(2,q)$, $q$ even, with $q+1$ singular conics and empty base.
\end{corollary}

\begin{proof}
If $q>2$ then a pencil $\cP$ with $q+1$ singular conics corresponds to a solid $S \in \Omega_1 \cup \ldots \cup \Omega_4$. 
By the point-orbit distributions calculated above, $S$ meets $\cV(\bF_q)$ in at least one point, so Lemma~\ref{baseLemma} implies that $\cP$ has at least one base point. 
By Section~\ref{F2}, the result holds also for $q=2$.
\end{proof}

\section{Solids contained in at least one and at most $q$ hyperplanes of type $\mathcal{H}_3$} \label{section2}

In this section we classify the $K$-orbits of solids contained in at least one hyperplane of type $\cH_3$ and at most $q$ such hyperplanes.
That is, we treat the solids $S$ with hyperplane-orbit distribution $\operatorname{OD}_{K,4}(S)=[a_1,a_{2r},a_{2i},a_3]$ where $1 \leqslant a_3 \leqslant q$. 
The cases (i) $a_1 \neq 0$, (ii) $a_1=0$ and $a_{2r} \neq 0$, and (iii) $a_1=a_{2r}=0$ and $a_{2i} \neq 0$ are analysed separately in Sections \ref{ss4.1}, \ref{ss4.2} and \ref{ss4.3}, respectively. 
The following observation implies that $a_1+a_{2r}+a_{2i} \leqslant 3$ (and hence $a_3 \geqslant q-2$) in all cases.

\begin{lemma}\label{3singular}
A pencil containing a nonsingular conic contains at most three singular conics.
\end{lemma}

\begin{proof}
A pencil generated by $\cZ(f)$ and $\cZ(g)$, with $\cZ(g)$ nonsingular, contains a singular conic $\cZ(f+\lambda g)$ if and only if $\lambda$ is a root of a (certain) cubic in $\bF_q[X]$ (cf. Lemma~\ref{pencils}).
\end{proof}

\subsection{Solids contained in a hyperplane of type $\cH_1$} \label{ss4.1}

The stabiliser of a nonsingular conic $\cC$ in $\PG(2,q)$ has three orbits on lines, namely tangents to $\cC$, secants to $\cC$, and lines external to $\cC$. 
Hence, there are at most three $K$-orbits of solids contained both in a hyperplane of type $\cH_3$ (which corresponds to a nonsingular conic) and a hyperplane of type $\cH_1$ (which corrersponds to a double line).
Since the corresponding types of pencils have different numbers of base points, there are exactly three $K$-orbits. 
The following representatives are obtained using the nonsingular conic $\cC = \cZ(X_0X_1+X_2^2)$ and the double lines corresponding to the tangent $\cZ(X_0)$, the secant $\cZ(X_2)$ and the external line $\cZ(X_0+X_1+\sqrt{\gamma} X_2)$, where $\gamma$ is some fixed element of $\bF_q$ with $\operatorname{Tr}(\gamma^{-1})=1$ (cf. Lemma~\ref{quadratic}):
\begin{equation} \label{567reps}
\Omega_5: \begin{bmatrix} \cdot&x&y\\x&z&t\\y&t&x \end{bmatrix}, \quad
\Omega_6: \begin{bmatrix} x&\cdot&y\\\cdot&z&t\\y&t&\cdot \end{bmatrix}, \quad
\Omega_7: \begin{bmatrix} x&y&z\\ y&x+\gamma y&t\\z&t&y \end{bmatrix} \text{ where } \operatorname{Tr}(\gamma^{-1})=1.
\end{equation}

\begin{table}[t!]
\center
\begin{tabular}{llll}
\toprule
Orbit & Point-orbit distribution & Hyperplane-orbit distribution & Stabiliser \\
\midrule
$\Omega_5$ & $[1,q+1,q^2-1,q^3]$ & $[1,0,0,q]$ & $E_q^2 : C_{q-1}$ \\
$\Omega_6$ & $[2,q+1,q^2+q-2,q^3-q]$ & $[1,1,0,q-1]$ & $C_{q-1}^2 : C_2$ \\
$\Omega_7$ & $[0,q+1,q^2+q,q^3-q]$ & $[1,0,1,q-1]$ & $D_{2(q+1)} \times C_{q-1}$ \\
\bottomrule
\end{tabular}
\caption{Data for Lemma~\ref{567hod}.}
\label{567table}
\end{table}

\begin{lemma} \label{567hod}
The point-orbit distributions, hyperlane-orbit distributions, and stabilisers of solids of types $\Omega_5$, $\Omega_6$ and $\Omega_7$ are as in Table~\ref{567table}. 
In particular, these orbits are distinct from each other and from $\Omega_1,\dots,\Omega_4$.
\end{lemma}

\begin{proof}
Let $S_i \in \Omega_i$, $i \in \{5,6,7\}$, be the representatives given in \eqref{567reps}. 
The hyperplane-orbit distribution of $S_5$ is an immediate consequence of Lemma~\ref{pencils}, which implies that a conic $\cZ(\lambda X_0^2+ X_0X_1+X_2^2)$ in the pencil $\cP(S_5)$ cannot be singular. 
Similarly, a conic $\cZ(\lambda X_2^2+ X_0X_1+X_2^2)$ in $\cP(S_6)$ is singular if and only if $\lambda=1$, in which case one obtains the pair of real lines $\cZ(X_0X_1)$, both of which are tangents to the conic $\cZ(X_0X_1+X_2^2)$.
Finally, a conic $\cZ(\lambda( X_0^2+X_1^2+\gamma X_2^2)+X_0X_1+X_2^2)$ in $\cP(S_7)$ is singular if and only if $\lambda =\gamma^{-1}$, in which case one obtains the pair of conjugate imaginary lines $\cZ(X_0^2+\gamma X_0X_1+X_1^2)$. 
The hyperplane-orbit distributions imply that $\Omega_5$, $\Omega_6$ and $\Omega_7$ are distinct and do not belong to $\{\Omega_1,\dots,\Omega_4\}$. 

Next, we calculate the point-orbit distributions. 
The cubic surface $\Psi(S_5)=\cZ(X^3+Y^2Z)$ consists of $q^2+q+1$ points, being a cone with vertex a point and base a planar rational cubic curve. 
It meets the nucleus plane $\pi_n$ in the line $\ell : \cZ(X,Z)$, and $\cV(\bF_q)$ in its unique singular point $P = (0,0,1,0)$, i.e. the image of the base point of $\cP(S_5)$ under~$\nu$. 
The cubic surface $\Psi(S_6)=\cZ(XT^2+Y^2Z)$ consists of $q^2+2q+1$ points, since its point set is in one-to-one correspondence with the points on the hyperbolic quadric $\cZ(XT+YZ)$. 
It meets $\pi_n$ in the line $\cZ(X,Z)$, and $\cV(\bF_q)$ in the images of the two base points of $\cP(S_6)$.
Finally, $\Psi(S_7)=\cZ(\gamma X Y^2+\gamma Y Z^2 +X T^2+X^2Y +XZ^2+Y^3)$ consists of 
two lines in the plane $\cZ(Y)$ and $q^2$ additional points. 
It is disjoint from $\cV(\bF_q)$ and intersects $\pi_n$ in the line $\cZ(X,Y)$.

It remains to calculate the stabilisers. 
If an element of $K$ represented by a matrix $(g_{ij}) \in \GL(3,q)$ fixes $S_5$ then it must fix the point $P = S_5 \cap \cV(\bF_q)$ and the line $\ell = S_5 \cap \pi_n$ (both calculated above). 
This occurs if and only if $g_{12}=g_{13}=g_{23}=g_{32}=0$. 
An element of $K_P \cap K_\ell$ fixes $S_5$ if and only if it also maps the point $Q=(1,0,0,0)$ into $S_5$, since $S_5 = \langle P,Q,\ell \rangle$. 
This occurs if and only if also $g_{33}^2 = g_{11}g_{22}$. 
Factoring out scalars therefore gives $K_{S_5} \cong E_q^2 : C_{q-1}$. 
Since $\mathcal{P}(S_6)$ contains a unique double line $\mathcal{L}_1^2$ and a unique pair of real lines $\mathcal{L}_2\mathcal{L}_3$, its stabiliser in $\PGL(3,q)$ is equal to the stabiliser of $\mathcal{L}_1$ inside the stabiliser $C_{q-1}^2 : \operatorname{Sym}_3$ of $\{ \mathcal{L}_1, \mathcal{L}_2, \mathcal{L}_3 \}$. 
Hence, $K_{S_6} \cong C_{q-1}^2 : C_2$. 
Finally, a solid $S_7 \in \Omega_7$ is contained in a unique hyperplane $H_1$ of type $\cH_1$, which meets $\cV(\bF_q)$ in a conic $\cC$, and in a unique hyperplane $H_2$ of type $\cH_{2i}$, which meets $\cV(\bF_q)$ in a point $P \not \in H_1$. 
Therefore, $K_{S_7}$ is a subgroup of $K_\cC \cap K_P \cong \GL(2,q)$. 
Since $S_7$ is disjoint from $\cV(\bF_q)$, it meets the conic plane $\pi = \langle\cC\rangle$ in a line $\ell$ external to $\cC$. 
By considering the action of $K_{S_7}$ on $\pi$, we therefore deduce that $K_{S_7}$ is a subgroup of the stabiliser of $\ell$ in $K_\cC \cap K_P$, which has shape $D_{2(q+1)} \times C_{q-1}$. 
The fact that $K_{S_7}$ is equal to this group follows from the one-to-one correspondence between the hyperplanes of type $\cH_{2i}$ through $P$ and the lines external to $\cC$ in $\pi$. 
(Over the quadratic extension of $\PG(5,q)$, $\ell$ meets $\cC$ in a pair of conjugate points, and $H_2$ meets the Veronese surface in two conjugate conics which pass through $P$ and meet $\cC$ in those points, so $H_2$ is uniquely determined by $\ell$.)
\end{proof}

\begin{remark}\label{rem:Omega6}
\normalfont 
It follows from the first part of the proof of Lemma~\ref{567hod} that $\Omega_6$ can also be obtained by considering either (i) a pencil spanned by a nonsingular conic $\cC$ and a pair of two real lines tangent to $\cC$, or (ii) a pencil spanned by a pair of real lines and a double line meeting the pair in two distinct points.
\end{remark}

\subsection{Solids contained in a hyperplane of type $\cH_{2r}$ and no hyperplane of type $\cH_1$} \label{ss4.2}

If $S$ is a solid with hyperplane-orbit distribution $[0,a_{2r},a_{2i},a_3]$ where $a_{2r}>0$ and $1 \leqslant a_3 \leqslant q$, then we may assume without loss of generality that $\cP(S)$ is generated by a nonsingular conic $\cC$ and a pair of real lines $\cL_1\cL_2$. 
Let us encode the configuration $({\mathcal{C}},\cL_1,\cL_2)$ by the pair of integers $(k_1,k_2)$ where $k_i$ denotes the number of points in $\cL_i\cap \mathcal C$. 
The possible configurations are $(k_1,k_2) = (2,2)$, $(2,1)$, $(2,0)$, $(1,1)$, $(1,0)$ or $(0,0)$. 
By Remark~\ref{rem:Omega6}, we may ignore the case $(k_1,k_2)=(1,1)$.

\subsubsection{$(k_1,k_2)=(2,2)$} \label{kk=22} 

If $(k_1,k_2)=(2,2)$ then $\cP(S)$ has either three or four base points. 
Exactly one $K$-orbit arises from each of these two cases. 
In the case of three base points, this follows from the fact that the stabiliser of a nonsingular conic acts $3$-transitively on its points; in the case of four base points, it follows from the fact that the image of a frame of $\PG(2,q)$ under $\nu$ spans a solid.
The resulting orbits are
\begin{equation} \label{89reps}
\Omega_{8}: \begin{bmatrix}x&y&z\\y&t&z\\z&z&y\end{bmatrix}, \quad
\Omega_{9}: \begin{bmatrix}x&x&y\\x&z&t\\y&t&t\end{bmatrix}.
\end{equation}
Here the representative for $\Omega_8$ is obtained from the pencil generated by $\cC = \cZ(X_0X_1+X_2^2)$ and the pair of real lines $\cL_1=\cZ(X_0+X_2)$ and $\cL_2 = \cZ(X_1+X_2)$, which meet in the point $(1,1,1)$ on $\cC$. 
To obtain the representative for $\Omega_9$, note that the conic $\cC = \cZ(X_0(X_0+X_1) + \lambda X_2(X_1+X_2))$ is nonsingular for all $\lambda \not \in \{0,1\}$, by Lemma~\ref{pencils}. 
Fix $\cC$ by choosing such a $\lambda$, and then take the pair of real lines $\cL_1\cL_2 = \cZ(X_0(X_0+X_1))$, which meets $\cC$ in the four points 
\begin{equation} \label{S9base}
P_1 = (0,1,0), \; P_2 = (1,1,0), \; P_3 = (0,1,1), \; P_4 = (1,1,1).
\end{equation} 

\begin{lemma} \label{89hod}
The hyperplane-orbit distribution of a solid of type $\Omega_8$, respectively $\Omega_9$, is $[0,2,0,q-1]$, respectively $[0,3,0,q-2]$.
In particular, these orbits are distinct and do not belong to $\{\Omega_1,\dots,\Omega_7\}$.
\end{lemma}

\begin{proof}
Let $S_8$ and $S_9$ denote the representatives in \eqref{89reps}. 
A conic $\cZ(X_0X_1+X_2^2 + \lambda((X_0+X_2)(X_1+X_2)))$ in the pencil $\cP(S_8)$ is singular if and only if $\lambda=1$, by Lemma~\ref{pencils}, and setting $\lambda=1$ yields a pair of real lines. 
As noted above, a conic $\cZ(X_0(X_0+X_1) + \lambda X_2(X_1+X_2))$ in $\cP(S_9)$ is singular if and only if $\lambda \in \{0,1\}$, and both values produce pairs of real lines distinct from the chosen generator $\cZ(X_0(X_0+X_1))$. 
\end{proof}

\begin{remark}\label{rem:O8}
\normalfont 
If $S_8 \in \Omega_8$ then the second pair of real lines in $\cP(S_8)$ has $(k_1,k_2) = (2,1)$: it comprises the secant $\cZ(X_2)$ and the tangent $\cZ(X_0+X_1)$ to the generating nonsingular conic $\cZ(X_0X_1+X_2^2)$. 
Since the stabiliser of a nonsingular conic $\cC$ acts $3$-transitively on the points of $\cC$, this implies that $\Omega_8$ is the only $K$-orbit obtained from a pencil generated by a nonsingular conic $\cC$ and a real line pair consisting of a secant and a tangent to $\cC$ meeting at a point not on $\cC$. 
(Note that the above lines meet in the point $(1,1,0)$, which is not on $\cZ(X_0X_1+X_2^2)$.)
On the other hand, the three pairs of real lines in $\cP(S_9)$ all have $(k_1,k_2) = (2,2)$. 
\end{remark}

\begin{lemma} \label{89pod}
The point-orbit distribution of a solid of type $\Omega_{8}$ is $[3,1,q^2+2q-3,q^3-q]$. 
The point-orbit distribution of a solid of type $\Omega_{9}$ is $[4,1, q^2+3q-4,q^3-2q]$.
\end{lemma}

\begin{proof}
Consider again the solids $S_8$ and $S_9$ in \eqref{89reps}. 
The cubic surface $\Psi(S_8) = \cZ(XYT+XZ^2+Y^3+Z^2T)$ intersects the plane $\cZ(X)$ in a rational cubic curve with $q+1$ points, and the points of $\Psi(S_8)\setminus\cZ(X)$ comprise the set $\{(1,0,0,t): t \in \Fq \} \cup \{(1,1,1,t): t\in \Fq\} \cup \{(1,y,z,f(y,z)): y,z \in \Fq; \; y\neq z^2\}$, where $f(y,z) = (z^2+y^3)/(y+z^2)$, which has size $q^2+q$. 
It meets $\cV(\bF_q)$ in the image of the base of $\cP(S_8)$, and the nucleus plane in a unique point. 
The cubic surface $\Psi(S_9)= \cZ(Z(XT+Y^2)+XT^2+X^2T)$ meets the plane $\cZ(X)$ in two lines and contains $q^2+q$ additional points, namely those comprising the set $\{(1,0,z,0): z \in \Fq\} \cup \{(1,1,z,1): z \in \Fq\} \cup \{(1,y,g(y,t),t):y,t \in \Fq; t \neq y^2\}$ where $g(y,t) = (t+t^2)/(t+y^2)$. 
It meets $\cV(\bF_q)$ in the image of the base of $\cP(S_9)$, and the nucleus plane in a point. 
\end{proof}

\begin{lemma} \label{89stabs}
If $S_8 \in \Omega_8$ then $K_{S_8} \cong C_{q-1} \times C_2$. 
If $S_9 \in \Omega_9$ then $K_{S_9} \cong \operatorname{Sym}_4$. 
\end{lemma}

\begin{proof}
The solid $S_8 \in \Omega_8$ given in \eqref{89reps} contains exactly two pairs of real lines, namely $\mathcal{L}_1\mathcal{L}_2$ and $\mathcal{L}_1'\mathcal{L}_2'$ where $\mathcal{L}_1 = \mathcal{Z}(X_1+X_2)$, $\mathcal{L}_2 = \mathcal{Z}(X_0+X_2)$, $\mathcal{L}_1' = \mathcal{Z}(X_2)$ and $\mathcal{L}_2' = \mathcal{Z}(X_0+X_1)$. 
Note that $\mathcal{L}_1$ and $\mathcal{L}_2$ meet in a point $P = (1,1,1)$ which also lies on $\mathcal{L}_2'$, while $\mathcal{L}_1'$ and $\mathcal{L}_2'$ meet in a point $P' = (1,1,0)$ disjoint from $\mathcal{L}_1\mathcal{L}_2$. 
The stabiliser $G \leqslant \PGL(3,q)$ of $\mathcal{P}(S_8)$ therefore fixes both of $\mathcal{L}_1'$ and $\mathcal{L}_2'$, because $\mathcal{L}_1'$ meets $\mathcal{L}_1\mathcal{L}_2$ in the unique point $P$ while $\mathcal{L}_2'$ meets $\mathcal{L}_1\mathcal{L}_2$ in two points, $Q=(1,0,0)$ and $R=(0,1,0)$. 
Hence, it also fixes $\mathcal{L}_1\mathcal{L}_2$ and therefore $P$. 
That is, $G$ is equal to the stabiliser of $P$, $P'$ and $\{Q,R\}$. 
Since $P'$, $Q$ and $R$ are collinear, $G \cong C_{q-1} \times C_2$. 
Explicitly, $K_{S_8} \cong G$ is generated by the elements of $K$ represented by the matrices
\begin{equation} \label{Sigma8stab}
\left[
\begin{matrix}
0 & 1 & 0 \\
1 & 0 & 0 \\
0 & 0 & 1
\end{matrix}
\right]
\quad \text{and} \quad
\left[
\begin{matrix}
1 & 0 & \omega+1 \\
0 & 1 & \omega+1 \\
0 & 0 & \omega
\end{matrix}
\right], 
\text{ where } \langle \omega \rangle = \mathbb{F}_q^\times.
\end{equation}
If $S_9 \in \Omega_9$ then the base of $\mathcal{P}(S_9)$ is the frame of $\PG(2,q)$ given in  \eqref{S9base}, so $K_{S_4} \cong \operatorname{Sym}_4$. 
\end{proof}

\subsubsection{$(k_1,k_2)=(1,0)$} \label{sss422}

To prove that the configuration $(k_1,k_2)=(1,0)$ leads to a unique $K$-orbit, we consider extending the nonsingular conic $\cC$ to a conic in the quadratic extension $\PG(2,q^2)$ of $\PG(2,q)$. 
For clarity, we write $\overline{\cC}$ for the extension of $\cC$ to $\PG(2,q^2)$, and use the same `bar' notation for the corresponding extensions of other objects, in particular $\overline{\cL}_1$ and $\overline{\cL}_2$ for the pair of real lines $\cL_1$ and $\cL_2$. 
Let $\sigma \in \mathrm{P}\Gamma\mathrm{L}(3,q^2)$ be the Frobenius collineation of $\PG(2,q^2)$ induced by the automorphism $a \mapsto a^q$ of $\bF_{q^2}$. 
Since $\cL_2$ is external to $\cC$ (i.e. $k_2=0$), $\overline{\mathcal{L}_2}$ intersects $\overline{\mathcal{C}}$ in a pair of conjugate points $(\overline{P_2},\overline{P_2^\sigma})$. 
Let $P_1$ denote the unique point in which $\cL_1$ meets $\mathcal{C}$, and let $G_{\overline{\mathcal{C}}}\cong \PGL(2,q^2)$ denote the stabiliser of $\overline{\mathcal{C}}$ in $\PGL(3,q^2)$.
Consider another real point $R_1$ and pair of conjugate points $\overline{R_2}$ and $\overline{R_2^\sigma}$, associated with a second pair of real lines $\cL_1'\cL_2'$ with $(k_1,k_2)=(1,0)$. 
Let $\alpha$ denote the unique projectivity in $G_{\overline{\mathcal{C}}}$ mapping the triple $(P_1,\overline{P}_2,\overline{P_2^\sigma})$ to $(R_1,\overline{R}_2,\overline{R_2^\sigma})$.
Since $G_{\overline{\mathcal{C}}}$ acts sharply $3$-transitively on the points of $\overline{\mathcal{C}}$ and $\alpha\sigma\alpha^{-1}\sigma$ fixes the triple
$(P_1,\overline{P}_2,\overline{P_2^\sigma})$ pointwise, $\alpha$ commutes with $\sigma$ and therefore belongs to $\PGL(3,q)$. 
In other words, the stabiliser of $\cC$ in $\PGL(3,q)$ acts transitively on pairs of real lines meeting $\cC$ in the configuration $(k_1,k_2)=(1,0)$, so there is a unique $K$-orbit of solids arising from this configuration.
We denote this orbit by $\Omega_{10}$ and choose the representative
\[
\Omega_{10} : \begin{bmatrix}x&y&z\\y& y+\gamma t& t\\ z& t&y\end{bmatrix}, 
\quad \text{where} \quad \operatorname{Tr}(\gamma^{-1})=1, 
\]
obtained by taking $\cC = \cZ(X_0X_1+X_2^2)$, $\cL_1 = \cZ(X_1)$ and $\cL_2 = \cZ(X_0+X_1+\gamma X_2)$. 

\begin{lemma} \label{10dists}
A solid $S_{10} \in \Omega_{10}$ has point-orbit distribution $[1,1,q^2+2q-1,q^3-q]$, hyperplane-orbit distribution $[0,1,1,q-1]$, and stabiliser $K_{S_{10}} \cong C_{q-1} \times C_2$. 
In particular, $\Omega_{10} \not \in \{\Omega_1,\dots,\Omega_9\}$.
\end{lemma}

\begin{proof}
Let $S_{10}$ be the solid given above. 
The hyperplane-orbit distribution is calculated via Lemma~\ref{pencils}, and implies that $\Omega_{10}$ is distinct from all previously considered $K$-orbits. 
Explicitly, the only singular conic in $\cP(S_{10})$ other than $\cL_1\cL_2$ is the pair of imaginary lines $\cL_1'\cL_2' = \cZ(X_1^2+\gamma X_1X_2 + X_2^2)$. 
The cubic surface $\Psi(S_{10})= \cZ(T^2X+\gamma TXY+\gamma TZ^2+XY^2+T^3+TZ^2)$ meets the plane $\cZ(Y)$ in the union of the nonsingular conic $\cC' = \cZ(Y,TX+\gamma Z^2+T^2+Z^2)$ and the line $\cZ(Y,T)$, which is tangent $\cC'$. 
The remaining points of $\Psi(S_{10})$ comprise the set $\{(f(z,t),1,z, t): z,t \in \Fq\}$, $f(z,t) = (z^2(1+\gamma t)+1)/( t^2+\gamma t+1)$, which has size $q^2$. 
Moreover, $\Psi(S_{10})$ meets $\cV(\bF_q)$ in the (unique) point $(1,0,0,0)$, and the nucleus plane in the point $(0,0,1,0)$. 
To calculate the stabiliser, note that $\mathcal{L}_1'$ and $\mathcal{L}_2'$ meet in a point $P' = (1,0,0)$ which also lies on $\mathcal{L}_1$, while $\mathcal{L}_1$ and $\cL_2$ meet in a point $P = (\gamma,0,1)$ disjoint from $\mathcal{L}_1'\mathcal{L}_2'$. 
Extending to $\PG(2,q^2)$, we therefore obtain a pencil $\overline{\cP(S_{10})}$ of type $\Omega_8$. 
In particular, the stabiliser $G \leqslant \PGL(3,q^2)$ of $\overline{\mathcal{P}(S_{10})}$ is equal to the stabiliser of $\overline{P}$, $\overline{P}'$ and $\{\overline{Q},\overline{R}\} = \overline{\mathcal{L}}_2 \cap \overline{\mathcal{L}_1'\mathcal{L}_2'}$. 
Hence, $G \cong C_{q^2-1} \times C_2$ by Lemma~\ref{89stabs}, and comparing with \eqref{Sigma8stab} we see that over $\mathbb{F}_q$ we obtain $K_{S_{10}} \cong C_{q-1} \times C_2$.
\end{proof}

\subsubsection{$(k_1,k_2)=(1,2)$}

Next we consider the configuration $(k_1,k_2)=(1,2)$, namely the case in which $\mathcal{L}_1$ is a tangent to $\cC$ and $\mathcal{L}_2$ is a secant to $\mathcal{C}$. 
If the point $P = \cL_1 \cap \cL_2$ is not on $\cC$ then, by Remark \ref{rem:O8}, we obtain the $K$-orbit $\Omega_{8}$. 
Hence, we may assume that $P$ is on $\cC$, and since the stabiliser of a nonsingular conic acts $3$-transitively on the points of the conic, a unique $K$-orbit arises in this way. 
(Indeed, $2$-transitivity is sufficient to guarantee this.) 
We denote this $K$-orbit by $\Omega_{11}$ and choose the representative
\[
\Omega_{11}:\begin{bmatrix} x&y&z\\y&t&\cdot\\z&\cdot&y \end{bmatrix},
\] 
obtained by taking $\cC = \cZ(X_0X_1+X_2^2)$, $\cL_1 = \cZ(X_1)$ and $\cL_2 = \cZ(X_2)$. 

\begin{lemma} \label{s7}
A solid $S_{11} \in \Omega_{11}$ has point-orbit distribution $[2,1,q^2+q-2,q^3]$, hyperplane-orbit distribution $[0,1,0,q]$, and stabiliser $K_{S_{11}} \cong E_q : C_{q-1}$. 
In particular, $\Omega_{11} \not \in \{ \Omega_1,\dots,\Omega_{10} \}$.
\end{lemma}

\begin{proof}
Let $S_{11}$ be the solid given above. 
Lemma~\ref{pencils} implies that the pair of real lines $\cL_1\cL_2$ is the only singular conic in the pencil $\cP(S_{11})$, so the hyperplane-orbit distribution of $S_{11}$ is $[0,1,0,q]$. 
In particular, $\Omega_{11}$ is distinct from all of $\Omega_1,\dots,\Omega_{10}$. 
The cubic surface $\Psi(\Omega_{11})=\cZ(XYT+Y^3+Z^2T)$ intersects the plane 
$\cZ(Y)$ in the two lines $\cZ(Y,Z)$ and $\cZ(Y,T)$ and contains $q^2-q$ additional points, comprising the set $\{(x,1,z,(x+z^2)^{-1}):x,z\in \bF_q;\;x\neq z^2\}$. 
There are two points in $S_{11} \cap \cV(\bF_q)$, namely $P_1 = (1,0,0,0)$ and $P_2 = (0,0,0,1)$, and one point $Q = (0,0,1,0)$ in which $S_{11}$ meets the nucleus plane. 
The stabiliser $K_{S_{11}}$ certainly fixes $Q$ and $\{P_1,P_2\}$. 
However, $P_1$ is the image under $\nu$ of the point of intersection of $\cL_1\cL_2$, so $K_{S_{11}}$ must fix $P_1$ and $P_2$ pointwise. 
An element of $K_{P_1} \cap K_{P_2} \cap K_Q$ is represented by a matrix $(g_{ij}) \in \GL(3,q)$ with $g_{12}=g_{21}=g_{23}=g_{31}=g_{32}=0$. 
It fixes $S_{11}$ if and only if it also maps the point $R = (0,1,0,0)$ into $S_{11}$. 
This occurs if and only if also $g_{11}g_{22} = g_{33}^2$, so $K_{S_{11}} \cong E_q : C_{q-1}$.
\end{proof}

\subsubsection{$(k_1,k_2)=(2,0)$}

We now show that the configuration $(k_1,k_2)=(2,0)$ also produces exactly one new $K$-orbit. 
As in the case $(k_1,k_2)=(1,0)$, consider the extension $\overline{\cC}$ of the nonsingular conic $\cC$ to $\PG(2,q^2)$. 
The extension $\overline{\cL}_1$ of the secant line $\cC$ meets $\overline{\cC}$ in two $\bF_q$-rational points, and the extension $\overline{\cL}_2$ of the external line $\cL_2$ meets $\cC$ in two $\bF_{q^2}$-rational points which are conjugate under the Frobenius collineation $\sigma$ induced by the automorphism $a \mapsto a^q$ of $\bF_{q^2}$. 
These four points form a frame of $\PG(2,q^2)$, since they lie on $\overline{\cC}$. 
Any two such configurations are therefore $\PGL(3,q^2)$-equivalent, via a unique $\alpha \in \PGL(3,q^2)$. 
Verifying that $\alpha \sigma\alpha^{-1}\sigma$ fixes the frame obtained from $\cL_1\cL_2$ implies that $\alpha\in \PGL(3,q)$, cf. the case $(k_1,k_2)=(1,0)$.  
Hence, we obtain at most one $K$-orbit from the configuration $(k_1,k_2)=(2,0)$. 
We verify below that this orbit is distinct from all previously considered orbits, and therefore label it $\Omega_{12}$ and choose the representative
\[
\Omega_{12} : \begin{bmatrix}x&y&z\\y&t&\gamma y+z\\z&\gamma y+z&y\end{bmatrix}, 
\quad \text{where} \quad \operatorname{Tr}(\gamma^{-1})=1,
\]
obtained by taking $\cC = \cZ(X_0X_1+X_2^2)$, $\cL_1 = \cZ(X_2)$ and $\cL_2 = \cZ(X_0+X_1+\gamma X_2)$. 

\begin{lemma}\label{s15}
A solid of type $\Omega_{12}$ has point-orbit distribution $[2,1,q^2+q-2,q^3]$, hyperplane-orbit distribution $[0,1,0,q]$, and stabiliser $K_{S_{12}} \cong C_2^2$. 
In particular, $\Omega_{12} \not \in \{\Omega_1,\dots,\Omega_{11}\}$. 
\end{lemma}

\begin{proof}
The proof is similar to that of Lemma~\ref{s7} (for $\Omega_{11}$). 
Taking $S_{12}$ to be the solid defined above, Lemma~\ref{pencils} yields the hyperplane-orbit distribution. 
The cubic surface $\Psi(S_{12})$ meets the plane $\cZ(Y)$ in two lines and contains $q^2-q$ further points. 
It meets $\cV(\bF_q)$ in the two points $P_1=(1,0,0,0)$ and $P_2=(0,0,0,1)$, and the nucleus plane in the point $Q = (0,0,1,0)$. 
The stabiliser $K_{S_{12}}$ must fix $Q$ and $\{P_1,P_2\}$. 
It induces a permutation group of order $2$ on $\{P_1,P_2\}$ because e.g. the element of $K$ represented by the matrix obtained by swapping the first and second columns of the identity fixes $S_{12}$ and swaps $P_1$ and $P_2$. 
An element of $K_{P_1} \cap K_{P_2} \cap K_Q$ is represented by a matrix $(g_{ij}) \in \GL(3,q)$ with $g_{12}=g_{21}=g_{31}=g_{32}=0$, $g_{22}=g_{11}$ and $g_{23}=g_{13}$. 
It fixes $S_{12}$ if and only if it also maps the point $(0,1,0,0)$ into $S_{12}$, which occurs if and only if $g_{33}=g_{11}$ and $g_{13} \in \{0,\gamma g_{11}\}$. 
Factoring out scalars, we see that the kernel of the action of $K_{S_{12}}$ on $\{P_1,P_2\}$ also has order $2$. 
Therefore, $K_{S_{12}} \cong C_2^2$. 
The point- and hyperplane-orbit distributions of $S_{12}$ imply that $\Omega_{12}$ is distinct from all previously considered $K$-orbits, with the possible exception of $\Omega_{11}$. 
However, $K_{S_{12}} \cong C_2^2$ is not isomorphic to $K_{S_{11}} \cong E_q : C_{q-1}$ (for any $q$), so also $\Omega_{12} \neq \Omega_{11}$.
\end{proof}

\begin{remark} \label{o6Remark}
\textnormal{It is also possible to distinguish between the $K$-orbits $\Omega_{11}$ and $\Omega_{12}$ using their line-orbit distributions, rather than their stabilisers, as follows. 
As per \cite{lines}, a line of type $o_6$ is characterised by having point-orbit distribution $[1,1,q-1,0]$. 
Considering again the solids $S_i \in \Omega_i$, $i \in \{11,12\}$, used above, we therefore see that in each case the only candidates for lines of type $o_6$ are the two lines $\langle Q,P_1 \rangle$ and $\langle Q,P_2 \rangle$, where $Q=(0,0,1,0)$ is the unique point in which $S_i$ meets the nucleus plane, and $P_1 = (1,0,0,0)$ and $P_2 = (0,0,0,1)$ are the two points of rank $1$ in $S_i$. 
Only one of these four lines has type $o_6$, namely $\langle Q,P_1 \rangle$ in the case $i=11$. 
Therefore, $S_{11}$ and $S_{12}$ have different line-orbit distributions, and so $\Omega_{11} \neq \Omega_{12}$. 
(We note also that there is a typo in \cite[Table~4]{lines}: the fifth column should say that a line of type $o_6$ contains {\em one} point of the nucleus plane. 
This is, however, clear from the representative given in \cite[Table~2]{lines}.)}
\end{remark}

\subsubsection{$(k_1,k_2)=(0,0)$} \label{subsec:(0,0)}

Finally, we show that the configuration $(k_1,k_2)=(0,0)$ also produces a unique $K$-orbit. 
It suffices to use an argument similar to the one used in the case $(k_1,k_2)=(2,0)$. 
This time, both $\cL_1$ and $\cL_2$ are external to $\cC$ and so both give rise to pairs of conjugate points (with respect to the Frobenius collineation $\sigma$). 
The four points again form a frame, so the same argument as before shows that at most one $K$-orbit arises. 
We denote this orbit by $\Omega_{13}$ and choose the representative
\[
\Omega_{13} : \begin{bmatrix} x&y&z\\y&\gamma x+y&t\\z&t&\gamma x+z \end{bmatrix}, 
\quad \text{where} \quad \operatorname{Tr}(\gamma)=1,
\]
obtained as follows. 
Consider the two pairs of imaginary lines $\cC_i=\cZ(f_i)$ where $f_1=\gamma X_0^2+X_0X_i+X_i^2$, $i \in \{1,2\}$. 
Then the pencil $\cP(S_{13})$ corresponding to the solid $S_{13}$ defined above is generated by $\cC_1$ and $\cC_2$. 
We must show that $\cP(S_{13})$ contains a nonsingular conic $\cC$ and a pair of real lines external to $\cC$. 
By Lemma~\ref{pencils}, the conic $\cZ(\lambda_1 f_1 + \lambda_2 f_2)$ is singular if and only if $\lambda_1=0$, $\lambda_2=0$ or $\lambda_1=\lambda_2$. 
Setting $\lambda_1=\lambda_2$ yields the pair of real lines $\cL_1=\cZ(X_1+X_2)$ and $\cL_2 = \cZ(X_0+X_1+X_2)$, both of which are external to every nonsingular conic in the pencil, by Lemma~\ref{quadratic}.

\begin{lemma} \label{13dists}
A solid $S_{13} \in \Omega_{13}$ has point-orbit distribution $[0,1,q^2+3q,q^3-2q]$, hyperplane-orbit distribution $[0,1,2,q-2]$, and stabiliser $K_{S_{13}} \cong C_2^2 : C_2$. 
In particular, $\Omega_{13} \not \in \{ \Omega_1,\dots,\Omega_{12} \}$.
\end{lemma}

\begin{proof}
Let $S_{13}$ be the solid defined above. 
The preceding discussion gives the hyperplane-orbit distribution, which implies that $\Omega_{13} \not \in \{ \Omega_1,\dots,\Omega_{12} \}$. 
The cubic surface $\Psi(S_{13})$ intersects the plane $\cZ(X)$ in three concurrent lines $\cZ(X,Y)$, $\cZ(X,Z)$ and $\cZ(X,Y+Z)$, and contains a further $q^2$ points, parameterised as $(1,y,z,f(y,z))$ where 
$f(y,z)=(\gamma+\gamma y +\gamma z+\gamma y^2+\gamma z^2+yz+y^2z+yz^2)^{1/2}$.  
It is disjoint from $\cV(\bF_q)$ and meets the nucleus plane in a unique point, so the point-orbit distribution of $S_{13}$ is $[0,1,q^2+3q,q^3-2q]$. 
It remains to calculate the stabiliser. 
As per the discussion preceding the lemma, if we extend $\cP(S_{13})$ to $\PG(2,q^2)$ we obtain a pencil with four base points comprising a frame $B=\{P_1,P_2,P_3,P_4\}$, say.  
This pencil has type $\Omega_9$, so its stabiliser $\overline{G} \leqslant \PGL(3,q^2)$ is isomorphic to $\operatorname{Sym}_4$, by Lemma~\ref{89stabs}. 
The stabiliser $G = \overline{G} \cap \PGL(3,q)$ of $\mathcal{P}(S_{13})$ is therefore a subgroup of $\operatorname{Sym}_4$. 
Now, $\mathcal{P}(S_{13})$ also contains a unique pair of real lines $\mathcal{L}_1\mathcal{L}_2$, and over $\mathbb{F}_q^2$ each of these lines meets two points of $B$, say $\overline{\mathcal{L}}_1 = \langle P_1,P_2 \rangle$ and $\overline{\mathcal{L}}_2 = \langle P_3,P_4 \rangle$. 
Since $G$ fixes $\mathcal{L}_1\mathcal{L}_2$, it fixes $\{ \{P_1,P_2\}, \{P_3,P_4\} \}$ over $\mathbb{F}_q^2$, and therefore induces a subgroup of the permutation group $H = \langle (P_1,P_2), (P_3,P_4), (P_1,P_3)(P_2,P_4) \rangle \cong C_2^2:C_2$ on $B$. 
Conversely, a calculation shows that $K_{S_{13}}$ contains the group generated by the elements of $K$ represented by the matrices
\[
\left[
\begin{matrix}
1 & 0 & 0 \\
1 & 1 & 0 \\
0 & 0 & 1
\end{matrix}
\right], 
\quad
\left[
\begin{matrix}
1 & 0 & 0 \\
0 & 1 & 0 \\
1 & 0 & 1
\end{matrix}
\right] 
\quad \text{and} \quad
\left[
\begin{matrix}
1 & 0 & 0 \\
0 & 0 & 1 \\
0 & 1 & 0
\end{matrix}
\right],
\]
which is isomorphic to $H$. 
We therefore conclude that $K_{S_{13}} \cong C_2^2 : C_2$.
\end{proof}

\subsection{Solids contained in no hyperplanes of type $\cH_1$ or $\cH_{2r}$} \label{ss4.3}

Of the solids $S$ with hyperplane-orbit distribution $\operatorname{OD}_{K,4}(S) = [a_1,a_{2r},a_{2i},a_3]$ where $1 \leqslant a_3 \leqslant q$, we have now classified those for which at most one of $a_1$ and $a_{2r}$ is $0$. 
It therefore remains to consider the case in which $\operatorname{OD}_{K,4}(S) = [0,0,a_{2i},a_3]$. 
This assumption implies, of course, that $a_{2i} \geqslant 1$, since $a_3 \leqslant q$. 
On the other hand, Lemma~\ref{prop}(ii) implies that $a_{2i} \leqslant 1$, since $a_{2r}=0$ and $b \geqslant 0$. 
Therefore, we must have $\operatorname{OD}_{K,4}(S)=[0,0,1,q]$. 
Note that this then forces $b=0$ in Lemma~\ref{prop}, so that $\cP(S)$ must have empty base.  
We claim that the hyperplane-orbit distribution $[0,0,1,q]$ gives rise to a unique $K$-orbit, with representative
\begin{equation} \label{14rep}
\Omega_{14} : \begin{bmatrix} x &y&\gamma x+y+\gamma t\\y&\gamma x+y&z\\\gamma x+y+\gamma t&z&t \end{bmatrix}, 
\quad \text{where} \quad \operatorname{Tr}(\gamma)=1.
\end{equation}
This solid, call it $S_{14}$, is obtained from the pencil generated by the nonsingular conic $\cZ(X_1^2+X_0X_2+\gamma X_2^2)$ and the pair of imaginary lines $\cL_1\cL_2 = \cZ(\gamma X_0^2 + X_0X_1 + X_1^2)$. 
Lemma~\ref{pencils} confirms that $\cP(S_{14})$ contains no other singular conics, and so $S_{14}$ has the desired hyperplane-orbit distribution; it also has empty base, since the unique real point $(0,0,1)$ on $\cL_1\cL_2$ does not lie on any of the nonsingular conics.

\begin{lemma} \label{14dists}
A solid of type $\Omega_{14}$ has point-orbit distribution $[0,1,q^2+q,q^3]$ and hyperplane-orbit distribution $[0,0,1,q]$.  
In particular, $\Omega_{14} \not \in \{ \Omega_1,\dots,\Omega_{13} \}$.
\end{lemma}

\begin{proof} 
It remains to calculate the point-orbit distribution. 
Taking $S_{14} \in \Omega_{14}$ as above, we calculate that the cubic surface $\Psi(S_{14})$ meets the plane $\cZ(X)$ in the line $\cZ(X,Y)$ and contains a further $q^2$ points, parameterised as $(1,y,f(y,t),t)$ with
$f(y,t)=(\gamma^2t^2+\gamma y t^2+ \gamma+ \gamma y + \gamma t+ \gamma y^2+ ty +ty^2+y^3)^{1/2}$. 
It is disjoint from $\cV(\bF_q)$ (since $\cP(S_{14})$ has empty base) and meets the nucleus plane in one point.
\end{proof}

We now show that all solids with hyperplane-orbit distribution $[0,0,1,q]$ belong to the $K$-orbit $\Omega_{14}$, before finally calculating the stabiliser of such a solid. 
 
\begin{lemma}
The solids with hyperplane-orbit distribution $[0,0,1,q]$ form one $K$-orbit.
\end{lemma}

\begin{proof}
Let $S$ be a solid with hyperplane-orbit distribution $[0,0,1,q]$, and let $\cL_1\cL_2$ be the unique pair of imaginary lines in the pencil $\cP(S)$. 
To prove the result, we consider the extension of $\cP(S)$ to $\PG(2,q^2)$. 
Since $\cL_1$ and $\cL_2$ are conjugate with respect to the Frobenius collineation $\sigma$ induced by the automorphism $a \mapsto a^q$ of $\bF_{q^2}$, let us relabel them as $\ell$ and $\ell^\sigma$. 
Choose a nonsingular conic $\cC$ in $\cP(S)$, and denote the extensions of $\cP(S)$, $\cC$, $\ell$ and $\ell^\sigma$ to $\PG(2,q^2)$ using a `bar' (as in previous such arguments). 
Recall from the discussion preceding Lemma~\ref{14dists} that $\ell$ and $\ell^\sigma$ are external to $\cC$, since $\cP(S)$ necessarily has empty base. 
We claim that $\overline{\ell}$ and $\overline{\ell^\sigma}$ are likewise external to $\overline{\cC}$. 
If $\overline{\ell}$ is a tangent to $\overline{\cC}$, meeting $\overline{\cC}$ in a point $P$, then $\overline{\ell^\sigma}$ is the tangent to $\overline{\cC}$ at the point $P^\sigma$. 
By the classification in Section~\ref{ss4.1}, specifically Remark~\ref{rem:Omega6}, the pencil $\overline{\cP(S)}$ then has type $\Omega_6$ (over $\bF_{q^2}$). 
In particular, $\{P,P^\sigma\}$ is the base of $\overline{\cP(S)}$, and the line $\langle P,P^\sigma \rangle$ is its unique double line. 
However, this line is fixed by $\sigma$, so we have a contradiction. 
If $\overline{\ell}$ is a secant to $\overline{\cC}$ then it meets $\overline{\cC}$ in a pair of conjugate points $\{P,P^\sigma\}$, and $\overline{\ell^\sigma}$ is also a secant, meeting $\overline{\cC}$ in another pair of conjugate points $\{Q,Q^\sigma\}$. 
These four points are distinct because the point of intersection of $\ell$ and $\ell^\sigma$ does not belong to $\cC$, so it follows from Section~\ref{kk=22} that $\overline{\cP(S)}$ has type $\Omega_9$. 
However, the conic comprising the pair of lines $\langle P,Q \rangle$ and $\langle P^\sigma,Q^\sigma \rangle$ then belongs to $\overline{\cP(S)}$, a contradiction since this line pair is fixed by $\sigma$. 
Hence, $\ell$ and $\ell^\sigma$ external to $\cC$ as claimed. 
Section~\ref{subsec:(0,0)} therefore implies that $\overline{\cP(S)}$ has type $\Omega_{13}$. 
Now suppose that $S'$ is a second solid with hyperplane-orbit distribution $[0,0,1,q]$, and let $m$, $m^\sigma$ be the unique imaginary line pair in $\cP(S')$. 
Since $\overline{\cP(S')}$ also has type $\Omega_{13}$, there exists a projectivity $\alpha \in \PGL(3,q^2)$ mapping $S$ to $S'$. 
Choose two points $R_1$ and $R_2$ on $\overline{\ell}$ that do not belong to $\overline{\ell^\sigma}$. 
Then $\Lambda = (R_1,R_2,R_1^\sigma,R_2^\sigma)$ is a frame of $\PG(2,q^2)$, mapped by $\alpha$ to a frame $(W_1,W_2,W_1^\sigma,W_2^\sigma)$, where without loss of generality the points $W_1$ and $W_2$ are on $\overline{m} \setminus \overline{m^\sigma}$. 
The projectivity $\alpha \sigma \alpha^{-1} \sigma$ fixes $\Lambda$ pointwise, and so is equal to the identity element of $\PGL(3,q^2)$. 
Hence, $\alpha$ commutes with $\sigma$, and therefore belongs to $\PGL(3,q)$. 
In other words, there exists an element of $\PGL(3,q)$ mapping $\cP(S)$ to $\cP(S')$, and so the solids $S$ and $S'$ belong to the same $K$-orbit.
\end{proof}

\begin{lemma}
If $S_{14} \in \Omega_{14}$ then $K_{S_{14}} \cong C_4$.
\end{lemma}

\begin{proof}
Let $\ell$ and $\ell^\sigma$ be the unique pair of imaginary lines in $\cP(S_{14})$, where $\sigma$ is the Frobenius collineation of $\PG(2,q^2)$ induced by the automorphism $a \mapsto a^q$ of $\bF_{q^2}$. 
As explained above, the extension $\overline{\cP(S_{14})}$ of the pencil $\cP(S_{14})$ to $\PG(2,q^2)$ has type $\Omega_{13}$. 
The base $B$ of $\overline{\cP(S_{14})}$ comprises two distinct points $P$ and $Q$ on the line $\overline{\ell}$ and their conjugates $P^\sigma$ and $Q^\sigma$ on $\overline{\ell^\sigma}$.
By the proof of Lemma~\ref{13dists}, the stabiliser of $\overline{\cP(S_{14})}$ in $\PGL(3,q^2)$ is isomorphic to the permutation group $H = \langle (P,Q), (P^\sigma,Q^\sigma), (P,P^\sigma)(Q,Q^\sigma) \rangle \leqslant \operatorname{Sym}(B)$, which has order $8$. 
Now, observe that the projectivity inducing the permutation $(P,Q)$ does not belong to $\PGL(3,q)$, because if an element of $\PGL(3,q)$ swaps $P$ and $Q$ then it must also swap $P^\sigma$ and $Q^\sigma$. 
(Indeed, none of the given generators of $H$ are realised over $\bF_q$.) 
Therefore, the stabiliser of $\cP(S_{14})$ in $\PGL(3,q)$ has order at most $4$. 
Conversely, if we take $S_{14}$ to be the solid defined in \eqref{14rep} then a calculation shows that $S_{14}$ is fixed by the subgroup of $K$ generated by the element of order $4$ represented by the matrix
\[
\left[
\begin{matrix}
1 & 0 & 0 \\
1 & 1 & 0 \\
0 & \gamma^{-1} & 1
\end{matrix}
\right]
\]
We therefore conclude that $K_{S_{14}} \cong C_4$, as claimed. 
\end{proof}

\section{Solids contained in $q+1$ hyperplanes of type $\mathcal{H}_3$}\label{section3}

It remains to consider the possibility that a solid $S$ of $\PG(5,q)$ is contained in $q+1$ hyperplanes of type $\cH_3$, or, equivalently, that the associated pencil of conics $\cP(S)$ contains $q+1$ nonsingular conics. 
We first establish the existence of such solids. 
Choose $b,c \in \bF_q$ such that the cubic $b\lambda^3+c\lambda+1$ has no roots over $\bF_q$. 
(For example, take the minimal polynomial of a primitive element $\alpha$ of the field extension $\bF_{q^3} / \bF_q$, scale it to make the constant term $1$, and then apply a coordinate transformation to eliminate the $\lambda^2$ term.) 
Lemma~\ref{pencils} shows that the pencil generated by $\cZ(X_0X_1 + X_2^2)$ and $\cZ(X_0X_2 + bX_1^2 + cX_2^2)$ contains $q+1$ nonsingular conics, and so we obtain the desired orbit of solids with hyperplane-orbit distribution $[0,0,0,q+1]$,
\begin{equation} \label{15rep}
\Omega_{15} : \begin{bmatrix}x&y&bz+cy\\y&z&t\\bz+cy&t&y\end{bmatrix}, 
\quad \text{where} \quad b\lambda^3+c\lambda+1 \text{ is irreducible over } \bF_q.
\end{equation}
By Lemma~\ref{prop}, a pencil of conics corresponding to a solid in this orbit has a unique base point. 

\begin{lemma}
A solid of type $\Omega_{15}$ has point-orbit distribution $[1,1,q^2-1,q^3+q]$ and hyperplane-orbit distribution $[0,0,0,q+1]$.
\end{lemma}

\begin{proof}
Let $S_{15}$ be the solid defined in \eqref{15rep}, for some fixed $b,c \in \bF_q$ such that $b\lambda^3 + c\lambda + 1$ is irreducible over $\bF_q$.  
It remains to calculate the point-orbit distribution of $S_{15}$. 
The cubic surface $\Psi(S_{15})$ intersects the plane $\cZ(Z)$ in a rational cubic curve consisting of $q+1$ points, and contains a further $q^2-q$ points, parameterised as $(f(y,t),y,1,t)$ with $f(y,t)=(b+cy^2+y^3)/(t^2+y)$ and $t^2\neq y$. 
It meets $\cV(\bF_q)$ in a unique point, and the nucleus plane in a unique point. 
Hence, the point-orbit distribution of $S_{15}$ is $[1,1,q^2-1,q^3+q]$.
\end{proof}

We now show that {\em every} solid with hyperplane-orbit distribution $[0,0,0,q+1]$ belongs to the $K$-orbit $\Omega_{15}$. 
We need to know the sizes of the following unions of $K$-orbits, which are calculated via the orbit--stabiliser theorem using the relevant stabilisers (from Table~\ref{invariantsTable}) and the fact that $|K| = |\PGL(3,q)| = q^3(q^3-1)(q^2-1)$:
\[
|\Omega_6\cup \Omega_7|=q^4(q^2+q+1), \; 
|\Omega_8\cup \Omega_{10}|=q^3(q^3-1)(q+1), \; 
|\Omega_9\cup \Omega_{13}|=\tfrac{1}{6}q^3(q^3-1)(q^2-1).
\]
Note also that $|\cH_1| = q^2+q+1$, $|\cH_{2r}| = \tfrac{1}{2}q(q+1)(q^2+q+1)$, $|\cH_{2i}| = \tfrac{1}{2}q(q-1)(q^2+q+1)$ and $|\cH_3| = q^5-q^2$. 
Write $\cH_2=\cH_{2r}\cup \cH_{2i}$ and note that $|\cH_2|=q^2(q^2+q+1)$.

\begin{lemma}\label{lem:q^2}
A hyperplane belonging to the $K$-orbit $\cH_3$ contains exactly $q^2$ solids that are contained in a hyperplane of type $\cH_1$ and in a hyperplane of type $\cH_2$.
\end{lemma}

\begin{proof} 
Since $\cH_3$ is a $K$-orbit, each of its hyperplanes contains the same number of solids that are contained in a hyperplane of type $\cH_j$ for both $j \in \{1,2\}$. 
Denote this number by $k$.
Let $H\in \cH_3$ and $H_1\in \cH_1$. 
By Section~\ref{ss4.1}, the solid $H\cap H_1$ belongs to one of the $K$-orbits $\Omega_5$, $\Omega_6$ or $\Omega_7$, and accordingly has hyperplane orbit distribution $[1,0,0,q]$,  $[1,1,0,q-1]$ or $[1,0,1,q-1]$ (by Lemma~\ref{567hod}). 
If a solid $H\cap H_2$ with $H_2\in \cH_2$ belongs to a hyperplane of type $\cH_1$, it therefore has type $\Omega_6$ or $\Omega_7$, and each such solid belongs to $q-1$ hyperplanes of type $\cH_3$.
Counting the flags $(H,S)$ where $H\in \cH_3$ and $S$ is a solid contained in a hyperplane of type $\cH_j$ for both $j \in \{1,2\}$ gives $|\cH_3|\cdot k =|\Omega_6\cup \Omega_7|\cdot (q-1)$, 
so $k=q^2$.
\end{proof}

\begin{lemma}\label{lem:card_15}
There are exactly $\frac{1}{3}q^3(q^3-1)(q^2-1)$ solids with hyperplane-orbit distribution $[0,0,0,q+1]$.
\end{lemma}

\begin{proof}
Consider a hyperplane $H$ of type $\cH_3$. 
If a solid contained in $H$ is contained in a hyperplane of type $\cH_1$, then it is contained in exactly one such hyperplane, by the classification in Section~\ref{section2}, so there are $|\cH_1| = q^2+q+1$ such solids in $H$. 
If a solid in $H$ is not contained in a hyperplane of type $\cH_1$, then it is contained in $i$ hyperplanes of type $\cH_2$ for some $i \in \{0,1,2,3\}$, by Lemma~\ref{3singular}. 
Let $n_i$ denote the number of solids contained in $H$ in each case. 
The total number of solids in $\PG(5,q)$ with hyperplane-orbit distribution $[0,0,0,q+1]$ is then equal to
\begin{equation} \label{15answer}
\frac{|\cH_3|\cdot n_0}{q+1},
\end{equation}
so we must calculate $n_0$. 
The total number of solids in $H$ is $(q^5-1)/(q-1)$, so $\sum_{i=0}^3 n_i = N - |\cH_1| = q(q^3+1)$. 
Now count the flags $(S,H')$ where $S$ is a solid in $H$ that is not contained in a hyperplane of type $\cH_1$ and $H'$ is a hyperplane of type $\cH_2$. 
By Lemma~\ref{lem:q^2}, we obtain $\sum_{i=1}^3 i\cdot n_i = |\cH_2|-q^2 = q(q^3+1)$. 
In particular, we have $\sum_{i=0}^3 n_i = \sum_{i=1}^3 i\cdot n_i$ and so $n_0 = n_2 + 2n_3$. 
Now, a solid contributing to $n_2$ belongs to $\Omega_8\cup \Omega_{10}$, so $n_2=(q-1)|\Omega_8\cup \Omega_{10}|/|\cH_3|=q(q^2-1)$. 
Similarly, a solid contributing to $n_3$ belongs to $\Omega_9\cup \Omega_{13}$, giving $n_3=(q-2)|\Omega_9\cup \Omega_{13}|/|\cH_3|=\tfrac{1}{6}q(q^2-1)(q-2)$. 
Therefore, $n_0 = n_2 + 2n_3 = \tfrac{1}{3}q(q+1)(q^2-1)$. 
Putting this into the expression in \eqref{15answer} completes the proof. 
\end{proof}

\begin{lemma}\label{thm:stab_15}
If $S_{15} \in \Omega_{15}$ then $K_{S_{15}} \cong C_3$.
\end{lemma}

\begin{proof}
To prove this, consider the {\em cubic} extension $\overline{\cP(S_{15})}$ of the pencil $\cP(S_{15})$, namely its extension to $\PG(2,q^3)$. 
Since $\cP(S_{15})$ contains no singular conics, $\overline{\cP(S_{15})}$ contains exactly three singular conics (cf. Lemma~\ref{3singular}), which must be conjugate under the Frobenius collineation $\sigma$ of $\PG(2,q^3)$ induced by the automorphism $a \mapsto a^q$ of $\bF_{q^3}$. 
In particular, these conics must all correspond to hyperplanes of $\PG(5,q^3)$ of the same type. 
According to the hyperplane-orbit distributions in Table~\ref{invariantsTable}, the only possibility is that $S_{15}$ has type $\Omega_9$ over $\bF_{q^3}$. 
Hence, by Lemma~\ref{89stabs}, the stabiliser $\overline{G} \leqslant \PGL(3,q^3)$ of $\overline{\cP(S_{15})}$ is isomorphic to the full permutation group of the four base points of $\overline{\cP(S_{15})}$. 
Only one of these base points, call it $Q$, is $\bF_q$-rational, since $\cP(S_{15})$ has a unique base point; the other three are conjugate under $\sigma$, so we may label them as $P$, $P^\sigma$, $P^{\sigma^2}$. 
The stabiliser $G \leqslant \PGL(3,q)$ of $\cP(S_{15})$ is therefore a subgroup of $\overline{G}_Q \cong \operatorname{Sym}_3$. 
We claim that $G$ induces a group of order $3$ on $\{ P, P^\sigma, P^{\sigma^2} \}$. 
If $\alpha \in G$ fixes one of these points, say $P$, but is not the identity, then it swaps $P^\sigma$ and $P^{\sigma^2}$ (and fixes $Q$), so $P^{\alpha\sigma} = P^\sigma$ and $P^{\sigma\alpha} = P^{\sigma^2}$, contradicting the fact that $\alpha$ commutes with $\sigma$. 
Therefore, $\alpha$ is the identity, and so $G$ induces no transpositions on $\{ P, P^\sigma, P^{\sigma^2} \}$. 
Conversely, consider the element $\beta \in \PGL(3,q^3)$ in the stabiliser of $\overline{\cP(S_{15})}$ corresponding to the $3$-cycle $(P, P^\sigma, P^{\sigma^2})$. 
Then $\beta$ commutes with $\sigma$ and so belongs to $G \leqslant \PGL(3,q)$. 
Hence, $G$ has order $3$.
\end{proof}

\begin{remark}
\textnormal{For reference, we also record a matrix representative $g \in \GL(3,q)$ for a generator of $K_{S_{15}}$, where $S_{15}$ is the solid given in \eqref{15rep}. 
If $q=2^n$ with $n$ even then we may choose $c=0$ and $b$ a non-cube. 
In this case, $g = \operatorname{diag}(1,\zeta,\zeta^2)$ where $\zeta \in \bF_q$ is a primitive third root of unity. 
If $n$ is odd then all elements of $\bF_q$ are cubes, so $c \neq 0$ and we can instead take $c=b$ after a change of variable $\lambda \rightarrow \sqrt{cb^{-1}} \lambda$. 
In this case, 
\[
g = \begin{bmatrix}1&0&0\\0&\zeta&b\\0&b&\zeta^2+b^2\end{bmatrix}, 
\quad \text{where} \quad
\zeta = b^{2^2} + b^{2^4} + \dots + b^{2^{n-1}}.
\]
}
\end{remark}

Lemmas \ref{lem:card_15} and \ref{thm:stab_15} together imply that there is a unique $K$-orbit of solids with hyperplane-orbit distribution $[0,0,0,q+1]$, as claimed (by the orbit--stabiliser theorem, since $|K|=q^3(q^3-1)(q^2-1)$).

\section{Solids in $\PG(5,2)$}\label{F2}

Tables~\ref{mainTableNew} and \ref{invariantsTable} are also correct for $q=2$, but some of the arguments in Sections~\ref{section1}--\ref{section3} do not apply in this case. 
For instance, the orbit $\Omega_1$ can no longer be obtained by considering two pairs of real lines meeting in a point, because a pencil of conics $\cP(S_1)$ corresponding to a solid $S_1 \in \Omega_1$ has a unique real line pair over $\bF_2$.  
Similarly, if $S_9 \in \Omega_9$ then $\cP(S_9)$ no longer contains any nonsingular conics, so the construction preceding Lemma~\ref{89hod} is not valid (but the generators given in Table~\ref{mainTableNew} are). 
Moreover, the point- and hyperplane-orbit distributions of $S_9$ now coincide with those of a solid $S_4 \in \Omega_4$, but the orbits of these solids can be distinguished either by their stabilisers, or by their line-orbit distributions: $S_4$ contains three lines of type $o_6$, while $S_9$ contains none (cf. Remark~\ref{o6Remark}).
In any case, it is straightforward to check the correctness of Tables~\ref{mainTableNew} and \ref{invariantsTable} for $q=2$ either by hand or via the FinInG package in GAP \cite{fining,GAP}. 
(Note that the descriptions of the stabilisers in Table~\ref{invariantsTable} simplify in the obvious ways when $q=2$, i.e. $C_{q-1}$ is the trivial group, $E_q \cong C_2$, and $\GL(2,q) \cong D_{2(q+1)} \cong \operatorname{Sym}_3$. 
Similarly, we necessarily have $\gamma=b=c=1$ in Table~\ref{mainTableNew}.)

\section{Comparison with Campbell's partial classification}\label{comparison1}

Campbell~\cite{campbell} provided a list of 17 ``classes'' and ``sets of classes'' of projectively equivalent pencils of conics in $\PG(2,q)$, $q$ even. 
His analysis divided the classes of pencils into the following sets: pencils with at least one double line (set~\RN{1}); pencils with no double lines and at least one real pair of lines (set~\RN{2}); pencils with no double lines, no real pairs of lines, and at least one conjugate imaginary pair of lines (set~\RN{3}); and pencils with no degenerate (singular) conics (set~\RN{4}).
The correspondence between our classification and Campbell's work is summarised in Table~\ref{comparison}. 
We remark that in the study of his set~\RN{3}, Campbell claimed that a pencil belonging to ``set~15'' has three imaginary pairs of lines and $q-2$ nonsingular conics. 
The non-existence of such a pencil was proved by Saniga~\cite{saniga} (and also follows from Table~\ref{invariantsTable}). 
Moreover, the existence of the $K$-orbit $\Omega_{14}$, whose elements have hyperplane-orbit distribution $[0,0,1,q]$, disproves Campbell's claim \cite[p.~405]{campbell} that there exists no pencil with a unique pair of imaginary conjugate lines and $q$ nonsingular conics.

\begin{table}[!t]
\center
\small
\begin{tabular}{ll}
\toprule
Class/Set & Orbit(s) \\
\midrule
Class~1 & $\Omega_3$ \\
Class~2 & $\Omega_5$ \\
Class~3 & $\Omega_1$ \\
Class~4 & $\Omega_2$ \\
Class~5 & $\Omega_7$ \\
Class~6 & $\Omega_6$ \\
\bottomrule
\end{tabular}
\begin{tabular}{ll}
\toprule
Class/Set & Orbit(s) \\
\midrule
Class~7 & $\Omega_9$ \\
Class~8 & $\Omega_{12}$ \\
Class~9 & $\Omega_8$ \\
Set~10 & $\Omega_9$, $\Omega_{12}$, $\Omega_{13}$ \\
Class~11 & $\Omega_{11}$ \\
Class~12 & $\Omega_4$ \\
\bottomrule
\end{tabular}
\begin{tabular}{ll}
\toprule
Class/Set & Orbit(s) \\
\midrule
Class~13 & $\Omega_{10}$ \\
Set~14 & $\Omega_{14}$ \\
Set~15 & $\Omega_{13}$ \\
Set~16 & $\Omega_{15}$ \\
Set~17 & $\Omega_{15}$ \\
& \\
\bottomrule
\end{tabular}
\caption{Correspondence between $\PGL(3,q)$-orbits of solids in $\PG(5,q)$ and Campbell's~\cite{campbell} ``classes'' and ``sets of classes'' of pencils of conics in $\PG(2,q)$, $q$ even.}
\label{comparison}
\end{table}

\section*{Acknowledgements}

The second author acknowledges the support of {\em The Scientific and Technological Research Council of Turkey} T\"UB\.{I}TAK (project no.~118F159).

\end{document}